\documentclass[11pt]{article}

\usepackage{amsfonts,amsmath,latexsym,color,epsfig}
\newtheorem{theorem}{Theorem}[section]
\newtheorem{lemma}[theorem]{Lemma}
\newtheorem{proposition}[theorem]{Proposition}
\newtheorem{conjecture}[theorem]{Conjecture}

\newtheorem{corollary}[theorem]{Corollary}

\newtheorem{claim}[theorem]{Claim}

\newtheorem{definition}[theorem]{Definition}

\newcommand{\E}[1]{\mathbb{E}\left[#1\right]}

\newcommand{\Var}[1]{\text{Var}\left(#1\right)}

\newcommand{\g}[2]{g_{_{#2,#1}}}
\newcommand{\gc}[1]{g_{_{#1}}}
\newcommand{\qed}{\hfill \ensuremath{\Box}}
\newcommand{\N}{\mathbb{N}}

\newcommand{\fP}{\mathcal{P}}

\newcommand{\size}[1]{\left| #1 \right|}

\newenvironment{proof}
      {\medskip\noindent{\bf Proof:}\hspace{1mm}}
      {\hfill$\Box$\medskip}

\input{epsf}

\makeatletter
\def\Ddots{\mathinner{\mkern1mu\raise\p@
\vbox{\kern7\p@\hbox{.}}\mkern2mu
\raise4\p@\hbox{.}\mkern2mu\raise7\p@\hbox{.}\mkern1mu}}
\makeatother

\begin{document}

\title{The Erd\H{o}s-Hajnal conjecture for rainbow triangles}
\date{}

\author{
Jacob Fox\thanks{
    Department of Mathematics,
    Massachusetts Institute of Technology,
    Cambridge, MA 02139-4307.
    Email: {\tt fox@math.mit.edu}.
    Research supported by a Simons Fellowship, by NSF grant DMS-1069197, by a Sloan Foundation Fellowship, and by an MIT NEC Corporation Award.}  \and 
Andrey Grinshpun\thanks{
    Department of Mathematics,
    Massachusetts Institute of Technology,
    Cambridge, MA 02139-4307.
    Email: {\tt agrinshp@math.mit.edu}.
    Research supported by a National Physical Science Consortium Fellowship.}
   \and J\'anos Pach\thanks{EPFL, Lausanne and Courant Institute, New York, NY. Supported by Hungarian Science Foundation EuroGIGA
Grant OTKA NN 102029, by Swiss National Science Foundation Grants 200020-144531 and 200021-137574, and by NSF Grant CCF-
08-30272. Email: {\tt  pach@cims.nyu.edu}.} 
} 

\maketitle

\begin{abstract}
We prove that every $3$-coloring of the edges of the complete graph on $n$ vertices without a rainbow triangle contains a set of order $\Omega\left(n^{1/3}\log^2n\right)$ which uses at most two colors, and this bound is tight up to a constant factor. This verifies a conjecture of Hajnal which is a case of the multicolor generalization of the well-known Erd\H{o}s-Hajnal conjecture. We further establish a generalization of this result. For fixed positive integers $s$ and $r$ with $s \leq r$, we determine a constant $c_{r,s}$ such that the following holds. Every $r$-coloring of the edges of the complete graph on $n$ vertices without a rainbow triangle contains a set of order $\Omega\left(n^{r(r-1)/s(s-1)}(\log n)^{c_{r,s}}\right)$ which uses at most $s$ colors, and this bound is tight apart from the implied constant factor. The proof of the lower bound utilizes Gallai's classification of rainbow-triangle free edge-colorings of the complete graph, a new weighted extension of Ramsey's theorem, and a discrepancy inequality in edge-weighted graphs. The proof of the upper bound uses Erd\H{o}s' lower bound on Ramsey numbers by considering lexicographic products of $2$-edge-colorings of complete graphs without large monochromatic cliques. 
\end{abstract}

\section{Introduction} 

A classical result of Erd\H{o}s and Szekeres \cite{ES35}, which is a quantitative version of Ramsey's theorem \cite{R30}, implies that every graph on $n$ vertices contains a clique or an independent set of order at least $\frac{1}{2}\log n$. In the other direction, Erd\H{o}s \cite{E47} showed that a random graph on n vertices almost surely contains no clique or independent set of order $2\log n$.

An {\it induced subgraph} of a graph is a subset of its vertices together with all edges with both
endpoints in this subset. There are several results and conjectures indicating that graphs
which do not contain a fixed induced subgraph are highly structured. In particular, the most famous
conjecture of this sort by Erd\H{o}s and Hajnal \cite{EH89} says that for each fixed graph $H$ there is $\epsilon=\epsilon(H)>0$ such that every graph $G$ on $n$ vertices which does not contain a fixed induced subgraph $H$ has a clique or independent set of order $n^{\epsilon}$. This is in stark contrast to general graphs, where the order of the largest guaranteed clique or independent set is only logarithmic in the number of vertices. 

There are now several partial results on the Erd\H{o}s-Hajnal conjecture. Erd\H{o}s and Hajnal \cite{EH89}  proved that for each fixed graph $H$ there is $\epsilon=\epsilon(H)>0$ such that every graph $G$ on $n$ vertices which does not contain an induced copy of $H$ has a clique or independent set of order $e^{\epsilon\sqrt{\log n}}$. Fox and Sudakov \cite{FS09}, strengthening an earlier result of Erd\H{o}s and Hajnal, proved that for each fixed graph $H$ there is $\epsilon=\epsilon(H)>0$ such that every graph $G$ on $n$ vertices which does not contain an induced copy of $H$ has a balanced complete bipartite graph or an independent set of order $n^{\epsilon}$. All graphs on at most four vertices are known to satisfy the Erd\H{o}s-Hajnal conjecture, and Chudnovsky and Safra \cite{CS08} proved it for the $5$-vertex graph known as the bull. Alon, Pach, and Solymosi \cite{APS01} proved that if $H_1$ and $H_2$ satisfy the Erd\H{o}s-Hajnal conjecture, then for every $v$ of $H_1$, the graph formed from $H$ by substituting $v$ by a copy of $H_2$ satisfies the Erd\H{o}s-Hajnal conjecture. The recent survey \cite{C12} discusses many further related results on the Erd\H{o}s-Hajnal conjecture. 

A natural restatement of the Erd\H{o}s-Hajnal conjecture is that for every fixed red-blue edge-coloring $\chi$ of a  complete graph, there is an $\epsilon=\epsilon(\chi)>0$  such that every red-blue edge-coloring of the complete graph on $n$ vertices without a copy of $\chi$ contains a monochromatic clique of order $n^{\epsilon}$. Indeed, for the graphs $H$ and $G$, we can color the edges red and the nonadjacent pairs blue. 

Erd\H{o}s and Hajnal also proposed studying a multicolor generalization of their conjecture. It states that for every fixed $k$-coloring of the edges of $\chi$ of a complete graph, there is an $\epsilon=\epsilon(\chi)>0$ such that every $k$-coloring of the edges of the complete graph on $n$ vertices without a copy of $\chi$ contains a clique of order $n^{\epsilon}$ which only uses $k-1$ colors. They proved a weaker estimate, replacing $n^{\epsilon}$ by $e^{\epsilon \sqrt{\log n}}$. Note that the case of two colors is what is typically referred to as the Erd\H{o}s-Hajnal conjecture. 

Hajnal \cite{H08} conjectured the following special case of the multicolor generalization of the Erd\H{o}s-Hajnal conjecture holds. There is $\epsilon>0$ such that every $3$-coloring of the edges of the complete graph on $n$ vertices without a rainbow triangle (that is, a triangle with all its edges different colors), contains a set of order $n^{\epsilon}$ which uses at most two colors.  We prove Hajnal's conjecture, and further determine a tight bound on the order of the largest guaranteed $2$-colored set in any such coloring. A {\it Gallai coloring} is a coloring of the edges of a complete graph without rainbow triangles, and a {\it Gallai $r$-coloring} is a Gallai coloring that uses $r$ colors.

\begin{theorem} \label{firsttheorem}
Every Gallai $3$-coloring on $n$ vertices contains a set of order $\Omega(n^{1/3}\log^2 n)$ which uses at most two colors, and this bound is tight up to a constant factor. 
\end{theorem}

To give an upper bound, we use lexicographic products. We will let $[m]=\{1,\ldots,m\}$ denote the set consisting of the first $m$ positive integers. 

\begin{definition}\label{lexprod}
Given edge-colorings $F_1$ of $K_{m_1}$ and $F_2$ of $K_{m_2}$ using colors from $R$, the \emph{lexicographic product coloring} $F_1 \otimes F_2$ of $E(K_{m_1m_2})$ is defined on any edge $e=\{(u_1,v_1),(u_2,v_2)\}$ (where we take the vertex set of $K_{m_1m_2}$ to be $[m_1] \times [m_2]$) to be $F_1(u_1,u_2)$ if $u_1 \neq u_2$, and otherwise $v_1 \neq v_2$ and it is defined to be $F_2(v_1,v_2)$.
\end{definition}
That is, there are $m_1$ disjoint copies of $F_2$ and they are connected by edge colors defined by $F_1$.

The upper bound in Theorem \ref{firsttheorem} is obtained by taking the lexicographic product of three $2$-edge-colorings of the complete graph on $n^{1/3}$ vertices, where each pair of colors is used in one of the colorings, and the largest monochromatic clique in each of the colorings is of order $O(\log n)$. A simple lemma in the next section shows that, in a lexicographic product coloring $F = F_1 \otimes F_2$, the largest set of vertices using only colors red and blue (for example) in $F$ has size equal to the product of the size of the largest set of vertices using only colors red and blue in $F_1$ with the size of the largest set of vertices using only colors red and blue in $F_2$. For any set $S$ of two of the three colors, the largest such set has order $O(n^{1/3})O(\log n)O(\log n)=O(n^{1/3}\log^2 n)$. 

In the other direction, we will utilize the following important structural result of Gallai \cite{Ga67} on edge-colorings of complete graphs without rainbow triangles. 

\begin{lemma}\label{lemmagallai} An edge-coloring $F$ of a complete graph on a vertex set $V$ with $|V| \geq 2$ is a Gallai coloring if and only if $V$ may be partitioned into nonempty sets $V_1,\ldots,V_t$ with $t > 1$ so that each $V_i$ has no rainbow triangles under $F$, at most two colors are used on the edges not internal to any $V_i$, and the edges between any fixed pair $(V_i,V_j)$ use only one color. Furthermore, any such substitution of Gallai colorings for vertices of a $2$-edge-coloring of a complete graph $K_t$ yields a Gallai coloring.
\end{lemma}

Gallai colorings naturally arise in several areas including in information theory \cite{KS00}, in the study of partially ordered sets, as in Gallai's original paper \cite{Ga67}, and in the study of perfect graphs \cite{CEL86}. 
There are now a variety of papers which consider Ramsey-type problems in Gallai colorings (see, e.g., \cite{CG83,FMO,GS10,GSSS10}). However, these works mainly focus on finding various monochromatic subgraphs in such colorings.

Because it may be of independent interest to the reader, we first present a particularly simple approach that will prove Hajnal's conjecture, but will not give tight bounds.

A graph is a {\it cograph} if it has at most one vertex, or if it or its complement is not connected, and all of its induced subgraphs have this property. In other words, the family of cographs consists of all those graphs that can be obtained from an isolated vertex by successively taking the disjoint union of two previously constructed cographs, $G_1$ and $G_2$, or by the join of them that we get by adding all edges between $G_1$ and $G_2$. It was shown by Seinsche \cite{Sei} that cographs are precisely those graphs which do not contain the path with three edges as an induced subgraph. It is easy to check by induction that every cograph is a perfect graph, that is, the chromatic number of every induced subgraph is equal to its clique number. 

\begin{proposition}\label{xyz312}
In any Gallai $3$-coloring of a complete graph, there is an edge-partition of the complete graph into three $2$-colored subgraphs, each of which is a cograph. 
\end{proposition}\vspace{-0.2cm}
\begin{proof}
This follows from Gallai's structure theorem by induction on the number of vertices. The result is trivial for edge-colorings of complete graphs with fewer than two vertices, which serves as the base case. Using Lemma \ref{lemmagallai}, we get a nontrivial vertex partition of the Gallai $3$-coloring of the complete graph into parts $V_1,\ldots,V_t$ such that only two colors appear between the parts. By the induction hypothesis, we can partition the edge-set of the complete graph on $V_i$ into three cographs, each which is two-colored. For the two colors that go between the parts, we take the graph which is the join of the cographs in each $V_i$, that is, add all edges between the parts, and for each of the other two pairs of colors, we just take the disjoint union of the cographs of those two colors from each part. Since the join or disjoint union of cographs are cographs, this completes the proof by induction. 
\end{proof}

The following corollary verifies Hajnal's conjecture and, apart from the two logarithmic factors, gives the lower bound in Theorem \ref{firsttheorem}.

\begin{corollary}\label{firstcorlast}
Every Gallai $3$-coloring of $E(K_n)$ contains a $2$-colored clique with at least $n^{1/3}$ vertices. 
\end{corollary} \vspace{-0.2cm}
\begin{proof}
Indeed, applying Proposition \ref{xyz312}, either the first cograph (which is $2$-colored) contains a clique of order $n^{1/3}$ (in which case we are done), or it contains an independent set of order $n^{2/3}$. In the latter case, this independent set of order $n^{2/3}$ in the first cograph contains in the second cograph a clique of order $n^{1/3}$ or an independent set (which is a clique in the third cograph) of order $n^{1/3}$. We thus get a clique of order $n^{1/3}$ in one of the three cographs, which is a $2$-colored set. \end{proof}

Improving the lower bound further to Theorem \ref{firsttheorem} appears to be considerably harder, and uses a different proof technique,  relying on a weighted version of Ramsey's theorem and a carefully chosen induction argument. The weighted version of Ramsey's theorem shows that if each vertex of a complete graph on $t$ vertices is given a positive red weight and a positive blue weight whose product is one,  then in any red-blue edge-coloring of $K_t$, there is a red clique $S$ and a blue clique $U$ such that the product of the red weight of $S$ (the sum of the red weights of the vertices in $S$) and the blue weight of $U$ (the sum of the blue weights of the vertices in $U$) is $\Omega\left(\log^2 t\right)$. Note that this extends the quantitative version of Ramsey's theorem as the case in which all the red and blue weights are one implies that there is a monochromatic clique of order $\Omega(\log t)$.

We further consider a natural generalization of this problem to more colors, and give a tight bound in the next theorem. In order to state the result more succinctly, we introduce some notation: for positive integers $r$ and $s$ with $s\leq r$, let 
$$
c_{r,s} = 
\begin{cases}
1 & \text{if $1=s<r$ or if $s=r-1$ and $r$ is even;}\\
s(r-s) & \text{if $1<s<r-1$;}\\
1+\frac{3}{r} & \text{if $s=r-1$ and $r$ is odd;}\\ 
0 & \text{if $s=r$.} 
\end{cases}
$$

\begin{theorem}\label{gentheorem}
Let $r$ and $s$ be fixed positive integers with $s \leq r$. Every $r$-coloring of the edges of the complete graph on $n$ vertices without a rainbow triangle contains a set of order $\Omega(n^{{s \choose 2}/{r \choose 2}}\log^{c_{r,s}} n)$ which uses at most $s$ colors, and this bound is tight apart from the constant factor.
\end{theorem}

We next give a brief discussion of the proof of Theorem \ref{gentheorem}. The case $s=r$ is trivial as the complete graph uses at most $r$ colors. The case $s=1$ is easy. Indeed, in this case, by the Erd\H{o}s-Szekeres bound on Ramsey numbers for $r$ colors, there is a monochromatic set of order $\Omega(\log n)$, where the implied positive constant factor depends on $r$. In the other direction, we give a construction which we conjecture is tight. 

The Ramsey number $r(t)$ is the minimum $n$ such that every $2$-coloring of the edges of the complete graph on $n$ vertices contains a monochromatic clique of order $t$. The bounds mentioned in the beginning of the introduction give $2^{t/2} \leq r(t) \leq 2^{2t}$ for $t \geq 2$. For $r$ even, consider a lexicographic product of $r/2$ colorings, each a $2$-edge coloring of the complete graph on $r(t)-1$ vertices with no monochromatic $K_t$.  This gives a Gallai $r$-coloring of the edges of the complete graph on $\left(r(t)-1\right)^{r/2}$ vertices with no monochromatic clique of order $t$. A similar construction for $r$ odd gives a Gallai $r$-coloring of the edges of the complete graph on $(t-1)\left(r(t)-1\right)^{(r-1)/2}$ vertices with no monochromatic clique of order $t$. The following conjecture which states that these bounds are best possible seems quite plausible.  It was verified by Chung and Graham \cite{CG83} in the case $t=3$.  

\begin{conjecture} 
Let $N(r,t)=\left(r(t)-1\right)^{r/2}$ for $r$ even and $N(r,t)=(t-1) \left(r(t)-1\right)^{(r-1)/2}$ for $r$ odd. For $n>N(r,t)$, every $r$-coloring of the edges of the complete graph on $n$ vertices has a rainbow triangle or a monochromatic $K_t$.  
\end{conjecture}

Having verified the easy cases $s=1$ and $s=r$ of Theorem \ref{gentheorem}, for the rest of the paper, we assume $1<s<r$.  A natural upper bound on the size of the largest set using at most $s$ colors comes from the following construction. We will let $[r]$ be the set of colors. Consider the complete graph on $[r]$, where each edge $P$ gets a positive integer weight $n_P$ such that the product of all $n_P$ is $n$. 
For each edge $P$ of this complete graph, we consider a $2$-coloring $c_P$ of the edges of the complete graph on $n_P$ vertices using the colors in $P$ and whose largest monochromatic clique has order $O(\log n_P)$, which exists by Erd\H{o}s lower bound \cite{E47} on Ramsey numbers. We then consider the Gallai $r$-coloring $c$ of the complete graph on $n$ vertices which is the lexicographic product of the ${r \choose 2}$ colorings of the form $c_P$. For each set $S$ of colors, the largest set of vertices in this edge-coloring of $K_n$ using only colors in $S$ has order $$\prod_{P \in S} n_P\prod_{|P \cap S|=1}O(\log n_P).$$ The order of the largest set using at most $s$ colors in coloring $c$ is thus the maximum of the above expression over all subsets $S$ of colors of size $s$. Therefore, we want to choose the various $n_P$ to minimize this maximum. For $s<r-1$, we give a second moment argument which shows that the best choice is essentially that the $n_P$ are all equal, i.e., $n_P=n^{1/{r \choose 2}}$ for all $P$. In this case, the above expression, for each choice of $S$, matches the claimed upper bound in Theorem \ref{gentheorem}. The case $s=r-1$ turns out to be more delicate. For $r$ even, the optimal choice turns out to be $n_P=n^{2/r}$ for $P$ an edge of a perfect matching of the complete graph with vertex set $[r]$, and otherwise $n_P=1$. For $r$ odd, we have three different edge weights. The graph on $[r]$ whose edges consist of those pairs with weight not equal to  $1$ consist of a disjoint union of a triangle and a matching with $(r-3)/2$  edges. The edges of the triangle each have weight $n^{1/r}(\log n)^{(r-3)/2r}$ and the edges of the matching each have weight $n^{2/r}(\log n)^{-3/r}$. It is straightforward to check that these choices of weights give the claimed upper bound in  Theorem \ref{gentheorem}. 

Similar to the case $r=3$ and $s=2$ mentioned above, using Gallai's structure theorem, we observe that, in any $r$-coloring of the edges of the complete graph on $n$ vertices without a rainbow triangle, the complete graph can be edge-partitioned into ${r \choose 2}$ subgraphs, each of which is a $2$-colored perfect graph. A simple argument then shows that there is a vertex subset of at least $n^{{s \choose 2}/{r \choose 2}}$ vertices which uses at most $s$ colors, which verifies the lower bound in Theorem \ref{gentheorem} apart from the logarithmic factors. Improving the lower bound further to Theorem \ref{gentheorem} is more involved, using a weighted version of Ramsey's theorem and a carefully chosen induction argument to prove this.

The rest of the paper is organized as follows. In the next section, we prove some basic properties of lexicographic product colorings. In Section \ref{sectsimple3}, we give simple proofs of lower and upper bounds in the direction of Theorem \ref{firsttheorem} which match apart from two logarithmic factors. In order to close the gap and obtain Theorem \ref{firsttheorem}, in Section \ref{sectweightrams} we prove a weighted extension of Ramsey's theorem. We complete the proof of Theorem \ref{firsttheorem} in Section \ref{tightfor3} by establishing a tight lower bound on the size of the largest $2$-colored set of vertices in any Gallai $3$-coloring of the complete graph on $n$ vertices.  The remaining sections are devoted to the proof of Theorem \ref{gentheorem}. In Section \ref{sectgeneralupperbound}, we prove the upper bound for Theorem \ref{gentheorem}. In Section \ref{sectweaklower}, we give a simple proof of a lower bound which matches Theorem \ref{gentheorem} apart from the logarithmic factors. In Section \ref{sectweight}, using the second moment method, we establish an auxiliary lemma which gives a tight bound on the minimum possible number of nonzero weights in a graph with non-negative edge weights such that no set of $s$ vertices contains sufficiently more than the average weight of a subset of $s$ vertices. We give the lower bound for Theorem \ref{gentheorem} in Section \ref{sectlowbound}, which completes the proof of this theorem. The proofs of some of the auxiliary lemmas which involve lengthy calculations are given in the appendix.  All logarithms in this paper are base $2$, unless otherwise indicated. All colorings are edge-colorings of complete graphs, unless otherwise indicated. For the sake of clarity of presentation, we systematically omit floor and ceiling signs whenever they are not crucial. We also do not make any serious attempt to optimize absolute constants in our statements and proofs.

\section{Lexicographic product colorings}

In this section, we will prove some simple results about lexicographic product colorings (Definition \ref{lexprod}). These will be useful in constructing examples of $r$-colorings that do not contain large vertex sets that use at most $s$ colors.

For such a lexicographic product coloring $F_1 \otimes F_2$ with $F_1$ on $m_1$ vertices and $F_2$ on $m_2$ vertices, we will view the vertex set interchangeably as $[m_1 \times m_2]$ and $[m_1] \times [m_2].$ For the sake of brevity, we often refer to a lexicographic product coloring as simply a product coloring.

\begin{definition}
For $F$ an edge-coloring of $K_n$ and $S \subseteq R$ a set of colors, we write that a set $Z$ of vertices is $S$-\emph{subchromatic} in $F$ if every edge internal to $Z$ takes colors (under $F$) only from $S$.
\end{definition}

When $F$ and $S$ are clear from context, we shall simply say that $Z$ is subchromatic. We will write $\g{F}{S}$ to be the size of the largest subchromatic set of vertices.

If $F$ is an edge-coloring constructed via a product of two other colorings $F_1,F_2$, then the next lemma allows us to determine $\g{F}{S}$ in terms of $\g{F_1}{S}$ and $\g{F_2}{S}$.

\begin{lemma}
For any $r$-colorings $F_1, F_2$ of $E(K_{n_1}),E(K_{n_2}),$ respectively, and any set $S \subseteq R$ of colors, $\g{F}{S}=\g{F_1}{S} \cdot \g{F_2}{S}$, where $F = F_1 \otimes F_2$.
\end{lemma}
\begin{proof}
Let $Z$ a set of subchromatic vertices in $F$ (so $Z \subseteq V(K_{n_1 \times n_2})$) be given. We will first show $\size{Z} \leq \g{F_1}{S} \cdot \g{F_2}{S}$.

Take $U \subseteq [n_1]$ to be the set of $u \in [n_1]$ such that there is some $v \in [n_2]$ with $(u,v) \in Z$; that is, $U$ is the subset of $[n_1]$ that is used in $Z$. For any $u \in [n_1]$, take $V_u\subseteq [n_2]$ to be the set of $v \in [n_2]$ such that $(u,v) \in Z$, that is, $V_u$ is the subset of $[n_2]$ that is paired with $u$ in $Z$. By construction, we have $Z = \bigcup_{u \in U} \{u\} \times V_u$.

Therefore, the set $U$ must be subchromatic in $F_1$, as given distinct $u_1,u_2 \in U$ there are $v_1,v_2$ so that $(u_1,v_1),(u_2,v_2) \in Z$, and hence: 
$$F_1(u_1,u_2)=F((u_1,v_1),(u_2,v_2))\in S.$$
Thus, $\size{U} \leq \g{F_1}{S}$.

Furthermore, given $u \in U$ we must have that $V_u$ is subchromatic in $F_2$, as given distinct $v_1,v_2 \in V_u$ we have that
$$F_2(v_1,v_2)=F((u,v_1),(u,v_2))\in S.$$ Therefore, $\size{V_u} \leq \g{F_2}{S}$.

Hence,
$$\size{Z} = \size{\bigcup_{u \in U} \{u\} \times V_u} =  \sum_{u \in U} \size{V_u} \leq \sum_{u \in U} \g{F_2}{S} = \size{U}\g{F_2}{S} \leq \g{F_1}{S} \cdot \g{F_2}{S}.$$
Since $Z$ was arbitrary, we get $\g{F}{S} \leq \g{F_1}{S} \cdot \g{F_2}{S}.$

We now prove that $\g{F}{S} \geq \g{F_1}{S} \g{F_2}{S}$, thus giving the desired result: take $U \subseteq [n_1]$ a subchromatic set under $F_1$ and $V \subseteq [n_2]$ a subchromatic set under $F_2$. We claim that $U \times V$ is subchromatic under $F$. Consider any distinct pairs $(u_1,v_1),(u_2,v_2) \in U \times V$. If $u_1 \neq u_2$ then 
$$F((u_1,v_1),(u_2,v_2)) = F_1(u_1,u_2) \in S,$$ and if $u_1 = u_2$ then 
$$F((u_1,v_1),(u_2,v_2))=F_2(v_1,v_2) \in S.$$
If we choose $U$ to have size $\g{F_1}{S}$ and $V$ to have size $\g{F_2}{S}$, we get $\g{F_1}{S} \cdot \g{F_2}{S} = \size{U \times V} \leq \g{F}{S}$.
\end{proof}

The next lemma states that the property of being a Gallai coloring is preserved under taking product colorings.

\begin{lemma}
If $F_1,F_2$ are Gallai $r$-colorings of $E(K_{n_1}),E(K_{n_2})$, respectively, then if $F=F_1 \otimes F_2$ then $F$ is a Gallai coloring.
\end{lemma}
\begin{proof}
Let any three vertices $u=(u_1,u_2),v=(v_1,v_2),w=(w_1,w_2) \in [n_1]\times [n_2]$ be given. We will show that they do not form a rainbow triangle under $F$. If $u_1=v_1=w_1$ then $F(u,v)=F_2(u_2,v_2),F(u,w)=F_2(u_2,w_2),F(v,w)=F_2(v_2,w_2)$ and so $u,v,w$ do not form a rainbow triangle by the assumption that $F_2$ is a Gallai coloring. If $u_1,v_1,w_1$ are pairwise distinct then $F(u,v)=F_1(u_1,v_1),F(u,w)=F_1(u_1,w_1),F(v,w)=F_1(v_1,w_1)$ and so $u,v,w$ do not form a rainbow triangle by the assumption that $F_1$ is a Gallai coloring. Otherwise, exactly one pair of $u_1,v_1,w_1$ are equal. Assume without loss of generality that $u_1=v_1$, $u_1 \neq w_1$, and $v_1 \neq w_1$. We have:
$$F(u,w)=F_1(u_1,w_1)=F_1(v_1,w_1) = F(v,w),$$
so again $u,v,w$ do not form a rainbow triangle.
\end{proof}

The following corollary states that we may take a product of any number of $2$-colorings and the result will be a Gallai coloring; since all $2$-colorings are Gallai colorings, it follows by induction from the previous lemma.
\begin{corollary}\label{isgallaicoloring}
If $F_1,\ldots,F_k$ are $2$-edge-colorings, then $F_1 \otimes \cdots \otimes F_k$ is a Gallai coloring.
\end{corollary}

\section{Simple bounds for three colors}
\label{sectsimple3}
In this section we will demonstrate simple upper and lower bounds in the case $r=3$ and $s=2$. We first apply the techniques of the previous section to demonstrate a Gallai $3$-coloring with no large $2$-colored vertex set.

\begin{theorem}
There is a Gallai $3$-coloring on $m$ vertices so that for every two colors $S \in {R \choose 2}$, every vertex set $Z$ using colors from $S$ satisfies $\size{Z} \leq (4/9+o(1))m^{1/3}\log^2m$.
\end{theorem}
\begin{proof}
Take $t = \lceil m^{1/3}\rceil$; then $t^3$ is at least $m$. For every pair of colors $P \in {R \choose 2}$, take $F_P$ to be a $2$-coloring of $E(K_t)$ using colors from $P$ so that the largest monochromatic clique has size at most $2\log t$. Such a coloring exists by the lower bound on Ramsey numbers proved by Erd\H{o}s and Szekeres in \cite{ES35}. We define $F$ a coloring on $t^3$ vertices by taking $F=F_{\{R_1,R_2\}} \otimes F_{\{R_2,R_3\}} \otimes F_{\{R_1,R_3\}}$ where $R_1,R_2,R_3$ are such that $R=\{R_1,R_2,R_3\}$. This is a Gallai coloring by Corollary \ref{isgallaicoloring}. Fixing any set $S$ of two colors, two of the above three colorings have $S$-subchromatic sets of size at most $2\log t$, and the remaining one has size $t$, so the size of the largest $S$-subchromatic set in $F$ is at most $t(2\log t)^2$. Since $S$ is arbitrary, the size of the largest $S$-subchromatic set for any $S \in {R \choose 2}$ is at most $t(2\log t)^2$.

Restricting $F$ to any $m$ vertices will be a $3$-Gallai coloring with no subchromatic set of size larger than $t(2\log t)^2$. Note that since $t = \lceil m^{1/3}\rceil$, we have $t = (1+o(1))m^{1/3}$, so
$$t(2\log t)^2 = (1+o(1))m^{1/3}(2\log(m^{1/3}))^2 = (4/9+o(1))m^{1/3}\log^2 m$$
\end{proof}

We now proceed to prove that any Gallai $3$-coloring on $m$ vertices contains a subchromatic set on two colors of size at least $m^{1/3}$. Indeed, the next theorem is a strengthening of this statement, as it states that the geometric average over $S \in {R \choose 2}$ of $\g{F}{S}$ must be at least $m^{1/3}$.

Since we have three colors, will refer to them as red, blue, and yellow. 

\begin{theorem}
\label{simplelowerr=3s=2}
For any Gallai $3$-coloring $F$ on $m$ vertices, $\prod_{S \in {R \choose 2}} \g{F}{S} \geq m.$
\end{theorem}
\begin{proof}
We proceed by induction on $m$ to prove the theorem.

Define $g$ to be the size of the largest subchromatic set using only the colors blue and yellow, $o$ to be the size of the largest subchromatic set using only the colors red and yellow, and $p$ to be the size of the largest subchromatic set using only the colors red and blue. (A note on nomenclature: $g$ stands for ``green," as blue and yellow form green when mixed. Similarly, $o$ stands for ``orange" and $p$ for ``purple.") We wish to show that $gop \geq n$.

If $m=1$, then $g=o=p=1$ and $gop=m$.

Otherwise, $m > 1$ and by the structure theorem for Gallai colorings there is a non-trivial partition of the vertex set into parts $V_1,\ldots,V_t$ and a pair of colors $Q \in {R \choose 2}$ satisfying that for any distinct $i,j \in [t]$ there is a $q \in Q$ so that every edge between $V_i$ and $V_j$ has color $q$. Take $m_i$ to be the size of $V_i$ Take $g_i$ to be the size of the largest set using only the colors blue and yellow from $V_i$, $o_i$ to be the size of the largest set using only the colors red and yellow from $V_i$, and $p_i$ to be the size of the largest set using only the colors red and blue from $V_i$. Without loss of generality we assume that $Q$ contains colors blue and yellow.

We have $g = \sum_i g_i$. Indeed, we may combine all the largest sets using colors blue and yellow from each $V_i$ to obtain a set of size $\sum_i g_i$ that only uses blue and yellow. 

Furthermore, $o \geq \max_i o_i$ and $p \geq \max_i p_i$. This gives:
$$gop = \sum_i g_iop \geq \sum_i g_io_ip_i \geq \sum_i m_i = m,$$
where the last inequality follows by the induction hypothesis applied to $F$ restricted to $V_i$.
\end{proof}

Note that we use $o \geq \max_i o_i, p \geq \max_i p_i$. It is on these inequalities that we will in the next sections gain multiple factors of $\log m$; if, for example, we find some set $U \subseteq [t]$ satisfying that for each distinct $i,j \in U$ the edges between $V_i,V_j$ are all yellow, then $o \geq \sum_{i \in U} o_i$. If it were the case that the $o_i,p_i$ were all pairwise equal, then we would get by the Erd\H{o}s-Szekeres bound for Ramsey numbers that $op = \Omega(\log^2 t \max_i o_ip_i)$; this motivates the approach in the next two sections, where we handle the general case in which it may not be true that the $o_i,p_i$ are all pairwise equal.

\section{A weighted Ramsey's theorem}
\label{sectweightrams}
In this section we will prove a version of Ramsey's theorem that will apply to graphs in which the weight of a vertex may depend on the color of the clique that contains the vertex. The next lemma is a convenient statement of a quantitative bound on the classical Ramsey's Theorem.

\begin{lemma}\label{RamseyTheorem}
In every $2$-coloring of the edges of $K_t$, for some $k$ and $\ell$ there is a red clique of order $k$ and a blue clique of order $\ell$ with $k \ell \geq \frac{1}{4}\log^2t$.
\end{lemma}
\begin{proof}
Take $k$ to be the order of the largest red clique and $\ell$ to be the order of the largest blue clique. We must have
$$t < R(k+1,\ell+1) \leq {k+\ell \choose k}.$$
It is routine to check that this implies $k \ell \geq \frac{1}{4}\log^2t$.
\end{proof}

For the rest of this paper, let $M:=2^{16}$. The following lemma, which we call the weighted Ramsey's theorem, states that if vertex $i$ contributes weight $\alpha_i$ to any red clique in which it is contained and weight $\beta_i$ to any blue clique in which it is contained, then we may give a lower bound for the product of the sizes of the largest (weighted) red and blue cliques.

\begin{lemma}
Given a $2$-coloring of the edges of a complete graph on $t$ vertices with $t \geq M$ and vertex weights $(\alpha_i,\beta_i)$, take $\gamma_i=\alpha_i\beta_i$ and $\gamma=\min_i \gamma_i$. There is a red clique $S$ and a blue clique $U$ with
$$\left(\sum_{s \in S} \alpha_s\right)\left(\sum_{u \in U} \beta_u\right) \geq \frac{\gamma}{32}\log^2t.$$
\end{lemma}
\begin{proof}
The proof will dyadically partition the vertices based on their pair of weights $(\alpha_i,\beta_i)$, and then apply the classical Erd\H{o}s-Szekeres bound on Ramsey numbers in the form of the previous lemma. That is, we will find a large set of vertices $A$ so that any two vertices in $A$ have similar values for $\alpha_i$ and $\beta_i$. Applying Lemma \ref{RamseyTheorem} to this set will give the desired result.

Take $\alpha=\max_i \alpha_i$ and $\beta=\max_i \beta_i$.

If $\alpha\beta \geq \frac{\gamma}{32}\log^2t$ we may take $S = \{i\}$ with $\alpha_i=\alpha$ and $U = \{j\}$ with $\beta_j=\beta$.
Otherwise, $\alpha \beta/\gamma < \frac{1}{32}\log^2t$. Observe that for each $i$ we have $\alpha_i \leq \alpha$, $\beta_i \leq \beta$, and $\alpha_i \beta_i \geq \gamma$.

This gives $\gamma/\beta \leq \alpha_i \leq \alpha$ and $\gamma/\alpha \leq \beta_i \leq \beta$. Note we may partition $[\gamma/\beta,\alpha]$ into $m_1\leq \log(\alpha\beta/\gamma)+1$ intervals $I_1,\ldots,I_{m_1}$ such that, within any interval $I_i$, we have $\sup(I_i)/\inf(I_i) \leq 2$. Similarly, we may partition $[\gamma/\alpha,\beta]$ into $m_2 \leq \log(\alpha\beta/\gamma)+1$ intervals $I_1',\ldots,I_{m_2}'$ with $\sup(I_i')/\inf(I_i') \leq 2$. By the pigeonhole principle there must be some pair $(j,j')$ such that, taking $A:=\{i : \alpha_i \in I_j, \beta_i \in I_{j'}'\}$, we have $\size{A} \geq t/(m_1m_2)$.

Applying the previous lemma to $A$, we get that there is a red clique $S$ of size $k$ and a blue clique $U$ of size $\ell$ with $k \ell \geq \frac{1}{4}\log^2 (t/(m_1m_2))$.

Note since $t \geq M$ we get $m_1m_2 \leq (\log(\frac{1}{32}\log^2t)+1)^2  = \log^2(\frac{1}{16}\log^2t) \leq t^{1/4}$. Therefore, we get
$$\frac{1}{4}\log^2(t/(m_1m_2)) \geq \frac{1}{4}\log^2(t^{3/4}) \geq \frac{1}{8}\log^2t.$$

Take $\alpha_A = \min_{i \in A} \alpha_i$ and $\beta_A = \min_{i \in A} \beta_i$. For any $i \in A$, $\alpha_i \in I_j$ and hence $\alpha_A \geq \alpha_i/2$. Similarly, for any $i \in A$ we have $\beta_A \geq \beta_i/2$. Therefore, fixing any $i \in A$, we get $\alpha_A\beta_A \geq \frac{\alpha_i}{2}\frac{\beta_i}{2} \geq \gamma/4$. Therefore,

$$\left(\sum_{s \in S} \alpha_s\right)\left(\sum_{u \in U} \beta_u\right) \geq \left(\sum_{s \in S} \alpha_A\right)\left(\sum_{u \in U} \beta_A\right) = k \alpha_A \ell \beta_A \geq k \ell \gamma/4 \geq \frac{\gamma}{32}\log^2t.$$
\end{proof}

Since in the statement of the weighted Ramsey's theorem we take $\gamma = \min_i \alpha_i\beta_i$, it provides good bounds when $\alpha_i\beta_i$ does not vary much between the vertices. Therefore, when we wish to use it in the upcoming sections, we will first dyadically partition the vertices based on $\alpha_i\beta_i$ and then apply the lemma to each partition.

Note that we chose $\gamma = \min_i \alpha_i\beta_i$. We may hope to be able to use other functions of $\alpha_i,\beta_i$ in this expression. However, it is not as robust as one may hope. In particular, we want to observe that the function $\alpha_i + \beta_i$ will not yield an analogous theorem, as if we have many vertices of weight $(0,1)$ and color all of the edges red, then the largest red clique has size $0$ and the largest blue clique has size $1$, but for each $i$ we have $\alpha_i+\beta_i = 1$. Fortunately, using $\alpha_i\beta_i$ will suffice for our purposes.

\section{Tight lower bound for three colors}\label{tightfor3}

In this section we will show that any Gallai $3$-coloring on $m$ vertices has a $2$-colored set of size $\Omega(m^{1/3}\log^2m)$. This matches the upper bound up to a constant factor.

We will refer to the three edge colors as red, blue, and yellow.

For the rest of this section, fix an integer $m \in \N$. We remark that in this section there is an inductive argument for which it is important to note that $m$ remains fixed throughout.

Let 

$$f(n):=
\left\{
\begin{array}{lll}
c\log^2(Cn) & \mbox{if } 0 < n \leq m^{4/9}\\
c^2\log^2(m^{4/9})\log^2(Cnm^{-4/9}) & \mbox{if } m^{4/9} < n \leq m^{8/9}\\
c^3\log^4(m^{4/9})\log^2(Cnm^{-8/9}) & \mbox{if } m^{8/9} < n \leq m,
\end{array} \right., $$
where $D=2^{2048}$, $C=2^{D^8},$ and $c=\log^{-2}(C^2) = D^{-16}/4$. We will have a further discussion about $f$ and its properties shortly. For now, simply note that $f(m) = \Omega(\log^6m)$.

We will prove the following theorem.
\begin{theorem}\label{theoremr=3s=2}
For any $n \in [m]$, a Gallai coloring $F$ on $n$ vertices has either $\max_S \g{F}{S} \geq m^{7/18}/8$ or $\prod_S \g{F}{S} \geq n f(n)$.
\end{theorem}

Before we prove Theorem \ref{theoremr=3s=2}, we show how it implies the existence of a large subchromatic set.

\begin{theorem}
Every Gallai $3$-coloring of $E(K_m)$ has a two colored set of size $\Omega(m^{1/3}\log^2m)$.
\end{theorem}
\begin{proof}
By Theorem \ref{theoremr=3s=2}, we have that either $\max_S \g{F}{S} \geq m^{7/18}/8 \geq \Omega(m^{1/3}\log^2m)$, or 
$$\prod_S \g{F}{S} \geq mf(m) = c^3 m \log^4(m^{4/9}) \log^2(C m^{1/9}) \geq c^3 m 2^{-6} (\log^4m)2^{-9}(\log^2m) = 2^{-15}c^3m\log^6m.$$
As we have a lower bound on the product of three numbers, one of these numbers must be at least the cubed root. Hence, $\max_S \g{F}{S} \geq 2^{-5}cm^{1/3}\log^2m \geq \Omega(m^{1/3}\log^2m)$, as desired.
\end{proof}

We will now proceed with a further discussion about $f$. We call $(0,m^{4/9}],(m^{4/9},m^{8/9}],(m^{8/9},m]$ the ``intervals of $f$." Note that on each interval, $f(n)=\gamma \log^2(\delta n)$ for some constants $\gamma,\delta$ (where $m$ is viewed as a constant). Intuitively, $C$ is large so that we avoid the range of values in which $\log$ is poorly behaved, and $c$ is small both so that we may assume $n$ is large and to make the transitions between intervals easier. $f$ was chosen so that it satisfies certain properties, the more interesting of which we explicitly enumerate below. All of these properties are formalizations of the statement ``$f$ does not grow too quickly."

\begin{lemma}
\label{factsaboutf}
If $m \geq C$, then the following statements hold about $f$ for any integer $n$ with $1 < n \leq m$.
\begin{enumerate}
\item For any $\alpha \in \left[\frac{1}{n},1\right], \, f(\alpha n) \geq \alpha f(n)$.
\item For any $\alpha_1,\alpha_2,\alpha_3 \in \left[ \frac{1}{n},1 \right]$ such that $\sum_i \alpha_i = 1$ we have, taking $n_i = \alpha_i n$, $$nf(n) - \sum_i n_i f(n_i) \leq \frac{8}{\log C}nf(n).$$
\item For $i \geq 0$ and $m^{7/18} \geq 2^j \geq 1$ we have $f(2^i)\log^2(D2^j) \geq 512 f(2^{i+\frac{8}{7}j})$.
\item For $1 \leq \tau \leq n \leq D^3 \tau$, we have $f(\tau) \geq f(n)/2$.
\item For any $\alpha \in \left[\frac{1}{n},\frac{1}{32}\right], \, f(\alpha n) \geq 16 \alpha f(n)$.
\end{enumerate}
\end{lemma}

These properties are collectively referred to as ``the facts about $f$" and are proved in Appendix \ref{appendixA}.

We now proceed with a proof of Theorem \ref{theoremr=3s=2}.

\noindent {\bf Proof of Theorem \ref{theoremr=3s=2}:} We proceed by induction on $n$. Define $g$ to be the size of the largest set in $F$ using only the colors blue and yellow, $o$ to be the size of the largest set in $F$ using only the colors red and yellow, and $p$ to be the size of the largest set in $F$ using only the colors red and blue. We wish to show that either $gop \geq nf(n)$ or $\max(g,o,p) \geq m^{7/18}/8$.

Our base cases are those $n$ for which $f(n) \leq 1$, as for these cases by Theorem \ref{simplelowerr=3s=2} $gop \geq n \geq nf(n)$. Since $c = \log^{-2}(C^2)$, any $n < C$ is a base case.

If we are not in a base case, we have $n \geq C$.

Since $F$ is a Gallai coloring, there is a non-trivial partition $V(K_n)=V_1 \cup \ldots \cup V_t$ 
with $|V_1| \geq \ldots \geq |V_t| \geq 1$ such that there is some $2$-coloring $\chi$ of $[t]$ such that for every distinct $i,j \in [t]$ and $u \in V_i$, $v \in V_j$, the color under $F$ of $\{u,v\}$ is $\chi(i,j)$.

Suppose without loss of generality that $\chi$ only uses the colors blue and yellow. The proof will split into three cases.

\noindent {\bf Cases 1 and 2, Preliminary Discussion:} These will be the cases in which $V_1$ has a substantial portion of the vertices. Let $U_1=V_1$, $U_2$ denote the union of $V_j$ over $j \neq 1$ such that $\chi(1,j)$ is yellow, and $U_3$ denote the union of $V_j$ over $j \neq 1$ such that $\chi(1,j)$ is blue. We have that $U_1,U_2,U_3$ is a non-trivial partition of $V$. Let $n_i=\size{U_i}$. Let $\alpha_i=\size{U_i}/n=n_i/n$ for $i=1,2,3$, so $\alpha_1+\alpha_2+\alpha_3=1$. 

For $i=1,2,3$, let $F_i$ be the coloring $F$ restricted to $U_i$. Let $g_i$ be the size of the largest subchromatic set in $F_i$ using only the colors blue and yellow, $o_i$ be the size of the largest subchromatic set in $F_i$ using only the colors red and yellow, and $p_i$ be the size of the largest subchromatic set in $F_i$ using only the colors red and blue. Suppose without loss of generality $n_2 \geq n_3$, so $\alpha_2 \geq (1-\alpha_1)/2$ and $\max(\alpha_1,\alpha_2) \geq 1/3$. By the induction hypothesis, for $i=1,2,3$, we have that either one of $g_i,o_i,p_i$ is at least $m^{7/18}/8$, in which case we may use $g \geq \max_i g_i,\, o \geq \max_i o_i,\, p \geq \max_i p_i$ to complete the induction, or
$$g_i o_i p_i \geq n_if(n_i).$$

Assume we are in this latter case. Since the $U_i$ are connected only by yellow and blue edges, we may take the largest subchromatic set using only yellow and blue from each $U_i$, giving $g \geq g_1+g_2+g_3$ (in fact, equality holds). Since $U_1$ and $U_2$ are connected with yellow edges, we may take the largest subchromatic set using only red and yellow from both $U_1$ and $U_2$, or we may simply take the largest such subchromatic set from $U_3$, so we get $o \geq \max(o_1+o_2,o_3)$. Similarly, $p \geq \max(p_1+p_3,p_2)$.

Note
$$gop \geq g_1op+g_2op+g_3op \geq g_1 (o_1+o_2)(p_1+p_3) + g_2(o_1+o_2)p_2 + g_3o_3(p_1+p_3) \geq $$
$$g_1(o_1+o_2)p_1+g_2(o_1+o_2)p_2+g_3o_3p_3 = g_1o_2p_1 + g_2o_1p_2 + \sum_{i=1}^3 g_io_ip_i.$$

We thus have

$$gop - \sum_{i=1}^3 g_io_ip_i \geq  g_1o_2p_1 + g_2o_1p_2 \geq 2\sqrt{(g_1o_2p_2)(g_2o_1p_1)}=2\sqrt{(g_1o_1p_1)(g_2o_2p_2)} \geq $$
$$2 \sqrt{(n_1f(n_1))(n_2f(n_2))}  \geq 2\sqrt{\alpha_1\alpha_2}n \sqrt{f(n_1)f(n_2)},$$
where the second inequality is an instance of the arithmetic-geometric mean inequality.

\noindent {\bf Case 1:} $\alpha_1,\alpha_2 \geq (\log C)^{-1/4}$. In this case, we have 
$$gop - \sum_{i=1}^3 g_io_ip_i \geq 2\sqrt{\alpha_1\alpha_2}n \sqrt{f(n_1)f(n_2)} \geq 2\alpha_1\alpha_2n f(n) \geq 2n f(n)/\sqrt{\log C} \geq$$ 
$$\frac{8}{\log C}nf(n) \geq  nf(n) - \sum_i n_if(n_i),$$
where the second inequality is by the first fact about $f$, the third inequality is by substituting lower bounds on $\alpha_1$ and $\alpha_2$, and the last inequality is by the second fact about $f$. Hence,  
$$gop \geq \sum_{i=1}^3 g_io_ip_i  + nf(n) - \sum_i n_if(n_i) \geq nf(n),$$
where the last inequality is by the induction on hypothesis applied to $U_i$ for $i=1,2,3$.  This completes this case. 

\noindent {\bf Case 2:} $\alpha_1 \geq (\log C)^{-1/4} \geq \alpha_2$. Before we proceed with this case, we prove a simple claim.

\begin{claim}
$nf(n)+n_1f(n_1) - 2n_1f(n) > 0$.
\end{claim}
\begin{proof}
Note $nf(n)-n_1f(n) = (1-\alpha_1) nf(n)$. Therefore,
$$n_1f(n_1)-n_1f(n) \geq \alpha_1n_1f(n)-n_1f(n) = \alpha_1^2nf(n)-\alpha_1nf(n) = -\alpha_1(1-\alpha_1)nf(n),$$
where the first inequality follows from the first fact about $f$.
From this we get
$$nf(n)+n_1f(n_1) - 2n_1f(n) \geq (1-\alpha_1)nf(n) - \alpha_1(1-\alpha_1)nf(n) = (1-\alpha_1)^2nf(n) > 0.$$
\end{proof}

In this case we have $\alpha_1 \geq 1-(\alpha_2+\alpha_3) \geq 1-2\alpha_2 \geq 1-2(\log C)^{-1/4} \geq 1/2$ and hence 

$$gop - \sum_{i=1}^3 g_io_ip_i \geq 2\sqrt{\alpha_1\alpha_2}n \sqrt{f(n_1)f(n_2)} \geq 8\alpha_1\alpha_2nf(n) \geq 4\alpha_2nf(n) \geq$$ 
$$2(\alpha_2 + \alpha_3)nf(n) = 2(n-n_1)f(n) \geq 2(n-n_1)f(n) - (nf(n) + n_1f(n_1) - 2n_1f(n)) = $$
$$nf(n)-n_1f(n_1) \geq nf(n) - \sum_i n_i f(n_i),$$
where the second inequality is by both the first fact about $f$ applied to $f(n_1)$ and the fifth fact about $f$ applied to $f(n_2)$, the third inequality is by $\alpha_1 \geq 1/2$, and the second-to-last one is by the claim.

Hence,  $$gop \geq \sum_{i=1}^3 g_io_ip_i + nf(n)-\sum_{i=1}^3 n_if(n_i) \geq nf(n),$$ where the last inequality is by the induction on hypothesis applied to $U_i$ for $i=1,2,3$.  This completes this case.

\noindent {\bf Case 3:} $\alpha_1 < (\log C)^{-1/4}$. This is the sparse case, when each part is at most a $(\log C)^{-1/4} = D^{-2}$ fraction of the total.

Take $n_i = \size{V_i}$. Take $F_i$ to be the coloring $F$ restricted to $V_i$. Take $g_i$ to be the size of the largest subchromatic set in $F_i$ using only the colors blue and yellow, $o_i$ to be the size of the largest subchromatic set in $F_i$ using only the colors red and yellow, and $p_i$ to be the size of the largest subchromatic set in $F_i$ using only the colors red and blue.

We reorder the $V_i$ so that if $i \leq j$ then $o_ip_i \leq o_jp_j$.

Take $\tau = \lfloor \log(2D^{-2}n) \rfloor$, so $\max_i n_i \leq (\log C)^{-1/4}n \leq D^{-2}n \leq 2^\tau \leq 2D^{-2}n$. Define, for $i \leq \tau$, $I_i := [2^i,2^{i+1}]$. Take $\Phi(i) = \{j : n_j \in I_i\}$. The $\Phi(i)$ are dyadically partitioning the indices; we will eventually use these partitions to construct sets to which we will apply the weighted Ramsey's theorem.

Note that $g = \sum_j g_j$, so we have $gop = \sum_j g_jop$.

We now present the idea behind the argument for the rest of this case. Fix $i$ so that $\Phi(i)$ has at least $2D2^{\frac{7}{8}(\tau-i)}$ elements and $i \geq \log(nm^{-7/18})$ (we will show that most vertices $v$ are contained in $V_j$ as $j$ varies over the $\Phi(i)$ that have this property). We will define a weighted graph whose vertices are the indices and whose coloring is $\chi$. Given an index $j$ its weight will be $(o_j,p_j)$. If we find a yellow clique in $\chi$ then the sum of the $o_j$ in the clique gives a lower bound on $o$, and similarly if we find a blue clique in $\chi$ then the sum of the $p_j$ in the clique gives a lower bound on $p$. We will apply the weighted Ramsey's theorem to half of the indices in $\Phi(i)$ (to the indices that are larger than the median of $\Phi(i)$, to be precise); from this, we will be able to conclude that if $j$ is an index smaller than the median, then $op/(o_jp_j) \geq D' f(n)/f(n_j)$ for some large constant $D'$ and so $g_jop \geq D'g_jo_jp_j f(n)/f(n_j) \geq D'n_j f(n)$. We now proceed with the argument.

When we count, we wish to omit parts $\Phi(i)$ that don't satisfy desired properties; take 
$$B':=\{i \leq \tau : \size{\Phi(i)} \leq 2D2^{\frac{7}{8}(\tau-i)}\},$$
$$B'':=\{i \leq \log(nm^{-7/18})\}.$$ Take $B = B' \cup B''$. We will show that a large fraction of the vertices are not contained in $V_j$ for $j \in \Phi(i)$ where $i$ ranges over $B$.

$$\sum_{i \in B'} \sum_{j \in \Phi(i)} n_j \leq \sum_{i \leq \tau} 2^{i+1}(2D 2^{\frac{7}{8}(\tau-i)}) = 4D2^{\frac{7}{8}\tau}\sum_{i \leq \tau} 2^{\frac{i}{8}}\leq$$ $$4D2^{\frac{7}{8}\tau}\frac{1}{2^{1/8}-1} \cdot 2^{(\tau+1)/8} \leq 8D\frac{1}{2^{1/8}-1} 2^\tau \leq 128 D 2^\tau \leq \frac{256}{D}n \leq n/4,$$
where the fourth inequality follows from $2^{1/8} \geq (1+1/16)$.

Note if $\sum_i g_i \geq m^{7/18}/8$ then we may complete the induction; assume this is not the case. In particular, we get $t \leq m^{7/18}/8$ (since $g_i \geq 1$). Therefore,
$$\sum_{i \in B''} \sum_{j \in \Phi(i)} n_j \leq \sum_{i \in B''} \sum_{j \in \Phi(i)}2nm^{-7/18} \leq 2tnm^{-7/18} \leq n/4.$$
Hence,
$$\sum_{i \in B} \sum_{j \in \Phi(i)} n_j \leq \sum_{i \in B'} \sum_{j \in \Phi(i)} n_j + \sum_{i \in B''} \sum_{j \in \Phi(i)} n_j \leq n/4 + n/4 \leq n/2.$$

As a corollary we get $\sum_{i \not \in B} \sum_{j \in \Phi(i)} n_j \geq n/2$.

For any fixed $i \leq \tau$ such that $i \not \in B$, take $\beta_i$ to be the median of $\Phi(i)$ (if $\Phi(i)$ has an even number of elements, take $\beta_i$ to be the larger of the two medians). Consider $\{(o_j,p_j) : j \in \Phi(i), j \geq \beta_i\}$. By $i \not \in B$, this has at least $D2^{\frac{7}{8}(\tau-i)} \geq M$ elements (recall from the weighted Ramsey's theorem that $M=2^{16}$), so we get by applying the weighted Ramsey's theorem to this set that $op \geq o_{\beta_i}p_{\beta_i}\log^2\left(D2^{\frac{7}{8}(\tau-i)}\right)/32$. Finally, observe that either one of the $o_j,p_j,g_j$ is at least $m^{7/18}/8$ in which case we may conclude the induction, or by the induction hypothesis we may assume $o_jp_jg_j \geq n_jf(n_j)$. Therefore,

$$\sum_{j \in \Phi(i)} g_jop \geq \sum_{j \in \Phi(i)} g_j o_{\beta_i}p_{\beta_i} \log^2\left(D2^{\frac{7}{8}(\tau-i)}\right)/32 \geq \sum_{j \in \Phi(i) : j \leq \beta_i} g_j o_{\beta_i}p_{\beta_i}\log^2\left(D2^{\frac{7}{8}(\tau-i)}\right)/32 \geq $$
$$\sum_{j \in \Phi(i) : j \leq \beta_i} g_j o_jp_j\log^2\left(D2^{\frac{7}{8}(\tau-i)}\right)/32 \geq \sum_{j \in \Phi(i) : j \leq \beta_i} n_jf(n_j)\log^2\left(D2^{\frac{7}{8}(\tau-i)}\right)/32 \geq $$
$$\sum_{j \in \Phi(i) : j \leq \beta_i} n_jf\left(2^{\log n_j}\right)\log^2\left(D2^{\frac{7}{8}(\tau-\log n_j)}\right)/32 \geq \sum_{j \in \Phi(i) : j \leq \beta_i} 16n_jf(2^\tau) \geq \sum_{j \in \Phi(i) : j \leq \beta_i} 8n_jf(n),$$
where the third inequality is by $o_j p_j \leq o_{j'}p_{j'}$ for $j \leq j'$, the fourth inequality is by the induction hypothesis applied to $V_j$, the sixth inequality is by the third fact about $f$, and the seventh inequality is by the fourth fact about $f$ and noting $2^\tau \geq D^{-3}n$.

We now consider for any set $J \subseteq \Phi(i)$:

$$\sum_{j \in J} n_j \geq 2^i \size{J}.$$
$$\sum_{j \in \Phi(i)} n_j \leq 2^{i+1} \size{\Phi(i)}.$$

This gives:
$$\frac{\sum_{j \in J} n_j}{\sum_{j \in \Phi(i)} n_j} \geq \frac{\size{J}}{2\size{\Phi(i)}}.$$

Noting that $\size{\{j \in \Phi(i) : j \leq \beta_i\}} \geq \size{\Phi(i)}/2$:

$$\sum_{j \in \Phi(i) : j \leq \beta_i} 8n_jf(n) \geq \frac{1}{4}\sum_{j \in \Phi(i)}8n_jf(n) = 2f(n)\sum_{j \in \Phi(i)}n_j.$$

Therefore,

$$gop \geq \sum_j g_j op \geq \sum_{i \leq \tau} \sum_{j \in \Phi(i)} g_jop \geq \sum_{i \leq \tau : i \not \in B} \sum_{j \in \Phi(i)} g_jop \geq \sum_{i \leq \tau : i \not \in B}2f(n)\sum_{j \in \Phi(i)}n_j =$$
$$2 f(n) \sum_{i \leq \tau : i \not \in B}\sum_{j \in \Phi(i)}n_j \geq 2f(n) \frac{n}{2} = nf(n).$$

We have thus concluded the induction. $\qed$

We informally refer to $B''$ in the above proof as large if a large fraction of the vertices are contained in a $V_j$ for $j \in \Phi(i)$ where $i$ ranges over $B''$. The case in which $B''$ was large easily implied the desired result. In extending this result in Section \ref{sectlower} to more colors, the primary difficulty is the following: when $s$ is not $2$, it is not obvious that there is a large $s$-colored set as a result of $B''$ being large.

\section{Upper bound for many colors}
\label{sectgeneralupperbound}
In this section we will give asymptotically tight upper bounds for how large of a subchromatic set must exist in an edge coloring on $m$ vertices. We will first show how to construct such colorings from weighted graphs with vertex set $R$, and then we will choose such graphs to finish the construction. The next theorem states that if we have a weighted graph on $r$ vertices with edge weights $w_P$, then we can find a coloring $F$ so that $\g{F}{S}$ is, up to logarithmic factors, $\prod_{P \subseteq S} w_P$.

\begin{lemma}
Given a weighted graph $(R,\fP)$ on $r$ vertices with integer edge weights $\{w_P\}_{P \in \fP}$, taking $m:=\prod_{P \in \fP} w_P$, there is a Gallai $r$-coloring on $m$ vertices so that for any $S \subseteq R$, the size of the largest subchromatic set with colors in $S$ is at most $\prod_{P \in \fP : P \subseteq S} w_P \cdot \prod_{P \in \fP : \size{P \cap S}=1} 2 \log w_P.$
\end{lemma}
\begin{proof}
We may define a Gallai $r$-coloring on $m$ vertices as follows: take $P_1,\ldots,P_k$ an arbitrary enumeration of $\fP$. For each edge $P$, take $F_P$ to be a $2$-coloring of $E(K_{w_P})$ using colors from $P$ so that the largest monochromatic clique has order at most $2 \log w_P$ (such a coloring exists by the Erd\H{o}s-Szekeres bound for Ramsey numbers \cite{ES35}). We define a coloring $F$ on $m$ vertices by
$$F=F_{P_1}\otimes F_{P_2} \otimes \cdots \otimes F_{P_{k}}.$$
$F$ is a Gallai coloring by Corollary \ref{isgallaicoloring}. Given any $S \subseteq R$, note that $\g{F_P}{S} = w_P$ if $P \subseteq S$, as $F_P$ uses only colors from $P$. If $\size{P \cap S} = 1$, then the largest subchromatic set in $F_P$ using colors from $P \cap S$ is at most $2 \log w_P$ by choice of $F_P$, so $\g{F_P}{S} \leq 2 \log w_P$. If $\size{P \cap S}=0$, then $\g{F_P}{S} = 1$ as any two distinct vertices are connected by an edge the color of which is not in $S$. Therefore,
$$\g{F}{S} = \prod_i \g{F_{P_i}}{S} \leq \prod_{P \in \fP : P \subseteq S}w_P \cdot \prod_{P \in \fP : \size{P \cap S} = 1} 2 \log w_P.$$
\end{proof}

The condition in the above lemma that the edge weights be integers is slightly cumbersome; we will now eliminate it.

\begin{lemma}
For any fixed integer $r \geq 3$, given a weighted graph $(R,\fP)$ on $r$ vertices with weights $\{w_P\}_{P \in \fP}$, taking $m:=\prod_{P \in \fP} w_P$, if $m$ is an integer and each $w_P$ satisfies $w_P \geq \omega(1)$, then there is a Gallai $r$-coloring on $m$ vertices so that for any $S \subseteq R$, the size of the largest subchromatic set is at most $(1+o(1))\prod_{P \in \fP : P \subseteq S} w_P \cdot \prod_{P \in \fP : \size{P \cap S}=1} 2 \log w_P.$
\end{lemma}
\begin{proof}
Take $w_P' = \lceil w_P \rceil$. Since $w_P \geq \omega(1)$, we get $w_P' \leq (1+o(1)) w_P$. We may apply the previous lemma to the $w_P'$ to get an $r$-Gallai coloring on $\prod_P w_P' \geq m$ vertices so that for any $S \subseteq R$ the size of the largest subchromatic set is at most
$$\prod_{P \in \fP : P \subseteq S} w_P' \cdot \prod_{P \in \fP : \size{P \cap S}=1} 2 \log w_P' \leq (1+o(1))\prod_{P \in \fP : P \subseteq S} w_P \cdot \prod_{P \in \fP : \size{P \cap S}=1} 2 \log w_P.$$
Restrict this coloring to any $m$ vertices; it is still a Gallai $r$-coloring and for any $S \subseteq R$ the size of the largest subchromatic set is at most $(1+o(1))\prod_{P \in \fP : P \subseteq S} w_P \cdot \prod_{P \in \fP : \size{P \cap S}=1} 2 \log w_P.$
\end{proof}

Now, if we wish to obtain colorings without large subchromatic sets, we need only construct appropriate weighted graphs. Intuitively, we would like to minimize the number of edges in such a graph (while still being able to maintain that all the $S \subseteq R$ have approximately the same value of $\prod_{P \subseteq S} w_P$), as every edge creates extra $\log$ factors. This observation motivates the following bounds.

\begin{theorem}
There is a Gallai $r$-coloring on $m$ vertices so that for any $S \in {R \choose s}$ the size of the largest subchromatic set is at most $(1+o(1))m^{{s \choose 2}/{r \choose 2}}\log^{c_{r,s}} m$, where

$$c_{r,s}=
\left\{
\begin{array}{cl}
s(r-s) & \mbox{if } s < r-1,\\
1 & \mbox{if } s=r-1 \mbox{ and r is even},\\
(r+3)/r & \mbox{if } s=r-1 \mbox{ and r is odd}.
\end{array}\right.
$$
\end{theorem}
\begin{proof}
If $s < r-1$, we may apply the previous lemma to a clique on $r$ vertices with edge weights $m^{1/{r \choose 2}}$. Any $S \subseteq R$ of size $s$ has ${s \choose 2}$ internal edges and $s(r-s)$ edges intersecting it in one vertex. By the previous lemma, we may find a Gallai $r$-coloring where the size of the largest subchromatic set is asymptotically at most:
$$m^{{s \choose 2}/{r \choose 2}}\left(2\log\left(m^{1/{r \choose 2}}\right)\right)^{s(r-s)} \leq m^{{s \choose 2}/{r \choose 2}}(\log m)^{s(r-s)}.$$

If $s=r-1$ and $r$ is even, we may consider a perfect matching on $r$ vertices where each edge has weight $m^{2/r}$; any subset of size $r-1$ contains $r/2-1$ edges and there is one edge with which it shares exactly one vertex. By the previous lemma, we may find a Gallai $r$-coloring where the size of the largest subchromatic set is asymptotically at most:
$$m^{(r/2-1)/(r/2)}2\log(m^{1/(r/2)}) \leq m^{(r/2-1)/(r/2)}\log m = m^{{s \choose 2}/{r \choose 2}}\log m.$$

If $s=r-1$ and $r$ is odd, we may consider a graph formed by taking the disjoint union of a triangle on $3$ vertices and a matching with $(r-3)/2$ edges. The edges of the triangle will each have weight $w_1:=m^{1/r}(\log m)^{(r-3)/2r}$ and the edges of the matching will each have weight $w_2 := m^{2/r}(\log m)^{-3/r}$. Note that the product of the weights is $w_1^3w_2^{(r-3)/2}=m$. Let $S \subseteq R$ of size $s=r-1$ be given.

If the vertex not contained in $S$ is part of the triangle then $S$ contains $(r-3)/2$ edges of weight $w_2$ and $1$ edge of weight $w_1$. Furthermore, there are two edges each of weight $w_1$ that $S$ intersects in one vertex. In the graph obtained from the previous lemma the size of the largest subchromatic set taking colors from $S$ is asymptotically at most:

\begin{eqnarray*}
w_1w_2^{(r-3)/2}(2\log w_1)^2 & = & m^{(r-2)/r}(\log m)^{-(r-3)/r}(2 \log(m^{1/r}(\log m)^{(r-3)/2r}))^2 \\
& \leq & m^{(r-2)/r}(\log m)^{-(r-3)/r}(\log m)^2 \\
& = & m^{{s \choose 2}/{r \choose 2}}(\log m)^{(r+3)/r}.
\end{eqnarray*}

If the vertex not contained in $S$ is part of the matching then $S$ contains $(r-5)/2$ edges of weight $w_2$ and $3$ edges of weight $w_1$. Furthermore, there is one edge of weight $w_2$ that intersects $S$ in one vertex. In the graph obtained from the previous lemma the size of the largest subchromatic set taking colors from $S$ is asymptotically at most:

\begin{eqnarray*}
w_1^3w_2^{(r-5)/2}(2 \log w_2) & = & m^{(r-2)/r}(\log m)^{3/r}(2 \log(m^{2/r}(\log m)^{-3/r})) \\
& \leq & m^{(r-2)/r}(\log m)^{3/r}(\log m) \\ & = & m^{{s \choose 2}/{r \choose 2}}(\log m)^{(r+3)/r}.
\end{eqnarray*}
\end{proof}

\section{Weak lower bound for many colors}
\label{sectweaklower}
We now provide a simple lower bound for the largest size of a subchromatic set in any $r$-coloring of $E(K_m)$ that shows our upper bounds are tight up to polylogarithmic factors; we show that any Gallai $r$-coloring on $m$ vertices contains a subchromatic set of size at least $m^{{s \choose 2}/{r \choose 2}}$. The following is a common generalization of H\"older's inequality that we will find useful.

\begin{lemma}\label{holder}
If $\mathcal{S}$ is a finite set of indices and for each $S \in \mathcal{S}$ we have $g_S$ is a function mapping $[t]$ to the non-negative reals, then
$$\prod_{S \in \mathcal{S}}\sum_i g_S(i) \geq \left(\sum_i \prod_{S \in \mathcal{S}} g_S(i)^{1/\size{\mathcal{S}}}\right)^{\size{\mathcal{S}}}$$
\end{lemma}

Using the above lemma, we will prove a lower bound on the product of the $\g{F}{S}$ for $F$ a Gallai $r$-coloring. This will easily imply the desired lower bound.

\begin{theorem}
For any Gallai $r$-coloring $F$ on $m$ vertices,
$$\prod_{S \in {R \choose s}} \g{F}{S} \geq n^{r-2 \choose s-2}.$$
\end{theorem}
\begin{proof}
Take $\gc{S} = \g{F}{S}$. We proceed by induction on $n$. If $n=1$, then each $\gc{S}$ is $1$ as is their product, while $n^{r-2 \choose s-2}$ is also $1$. If $n>1$, we may find some pair of colors $Q$ and some non-trivial partition of the vertices $V_1,\ldots,V_t$ such that for each pair of distinct $i,j$ in $[t]$, there is a $q \in Q$ so that all of the edges between $V_i$ and $V_j$ have color $q$.

Define, for $i \in [t]$, $F_i$ to be the restriction of $F$ to $V_i$. Take $\g{i}{S}:=\g{F_i}{S}$. By induction, for each $i$ we have $\prod_S \g{i}{S} \geq n_i^{r-2 \choose s-2},$ where $n_i = \size{V_i}$.

Note that if $Q \subseteq S$ then $\gc{S} \geq \sum_i \g{i}{S}$, since we may combine the largest subchromatic sets from each $F_i$. For every $S$ we have $\gc{S} \geq \max_i \g{i}{S}$, so

$$\prod_S \gc{S} \geq \left(\prod_{S : Q \subseteq S} \sum_i \g{i}{S} \right)\prod_{S : Q \not \subseteq S} \gc{S} \geq \left(\sum_i \prod_{S : Q \subseteq S}\g{i}{S}^{1/{r-2 \choose s-2}}\right)^{r-2 \choose s-2}\prod_{S : Q \not \subseteq S} \gc{S} =$$ 
$$\left(\sum_i \prod_{S : Q \subseteq S}\g{i}{S}^{1/{r-2 \choose s-2}} \prod_{S : Q \not \subseteq S}\gc{S}^{1/{r-2 \choose s-2}}\right)^{r-2 \choose s-2} \geq \left(\sum_i \prod_S \g{i}{S}^{1/{r-2 \choose s-2}}\right)^{r-2 \choose s-2} \geq \left( \sum_i n_i \right)^{r-2 \choose s-2} = n^{r-2 \choose s-2},$$
where the first inequality is by $\gc{S} \geq \sum_i \g{i}{S}$ if $Q \subseteq S$, the second inequality is by the preceding lemma and noting $\size{S} = {r-2 \choose s-2}$, the third inequality is by $\gc{S} \geq \g{i}{S}$, and the fourth inequality is by the induction hypothesis.
\end{proof}

Note that in proving this bound, if $\size{S \cap Q} = 1$ we simply use $\gc{S} \geq \g{F_i}{S}$. As in the $r=3,s=2$ case, if we can find a set of indices $V_{i_1},\ldots,V_{i_k}$ so that between any two of them the edges use the color contained in $S \cap Q$, we may obtain a stronger lower bound on $\gc{S}$.

We now conclude the argument.

\begin{theorem}
In any Gallai $r$-coloring $F$ on $m$ vertices, there is some $S \in {R \choose s}$ so that $\g{F}{S} \geq m^{{s \choose 2}/{r \choose 2}}$
\end{theorem}
\begin{proof}
By the previous theorem, $\prod_{S \in {R \choose s}} \g{F}{S} \geq m^{r-2 \choose s-2}$. As this is a product over ${r \choose s}$ numbers, there must be some $S$ with 
$$\g{F}{S} \geq m^{{r-2 \choose s-2}/{r \choose s}} = m^{{s \choose 2}/{r \choose 2}}.$$
\end{proof}

\section{Lower bound for many colors}\label{sectlower}
In this section we show that our upper bounds on sizes of subchromatic sets in Gallai colorings are tight up to constant factors (where we view $r$ and $s$ as constant).

\subsection{Discrepancy lemma in edge-weighted graphs} \label{sectweight}

The lemma in this subsection has the following form: either a given weighted graph has many edges of non-zero weight or it has some set $S$ of size $s$ whose weight is significantly larger than average. In the next subsection we will show how to reduce the problem of lower bounding the size of the largest subchromatic set in a Gallai $r$-coloring to a problem regarding the number of non-zero edges in a graph that doesn't contain vertex subsets $S$ whose weight is significantly larger than average, so this lemma will be useful.

\begin{lemma}\label{weightedgraphs}
Given weights $w_P$ for $P \in {R \choose 2}$ with $w_P \geq 0$, take $w = \sum_P w_P$. Take $a_0 = {r \choose 2}$ if $s < r-1$, $a_0 = r/2$ if $s=r-1$ and $r$ is even, and $a_0=(r+3)/2$ if $s=r-1$ and $r$ is odd. Either there are at least $a_0$ pairs $P$ with $w_P > 0$ or there is some $S \subseteq R$ of size $s$ satisfying 
$$\sum_{P \subseteq S} w_P \geq \left(1 + \left(4r {r \choose 2}^2\right)^{-1}\right)\frac{{s \choose 2}}{{r \choose 2}}w.$$
\end{lemma}

The proof of the above lemma uses elementary techniques along with the second moment method and is deferred to Appendix \ref{appendixB}.

\subsection{Proof of lower bound for many colors}\label{sectlowbound}

Let 
$$d=\frac{{r-2 \choose s-1}}{{r-2 \choose s-2}} = \frac{r-s}{s-1},$$
$$C=32r{r \choose 2}^3d,$$
$$\delta = \left(4{r-2 \choose s-2}C\right)^{-1},$$
$$\delta_0 = C^{-1}{r-2 \choose s-1}{r-2 \choose s-2}^{-1} \left({r \choose 2}+1\right)^{-1}=C^{-1}d \left({r \choose 2}+1\right)^{-1},$$
$$\delta_1 = 2^{-{r\choose 2}-2}\left(\delta_0^{-1}+1\right)^{-{r \choose 2}-1} {r-2 \choose s-1}^{-1},$$
$$c = (\delta/4)^2\delta_1^{1/d}.$$
$d$ is an appropriately chosen scaling factor; why it is appropriate will become evident later. $C$ should be thought of as a large constant, and $\delta,\delta_0,\delta_1,$ and $c$ should be thought of as small constants. We provide some bounds on the above; although we will not explicitly reference these, they are useful for verifying various inequalities:

$$r^{-1} \leq d \leq r,$$
$$C \leq 4r^8,$$
$$\delta \geq 2^{-8r},$$
$$\delta_0 \geq r^{-11}/4,$$
$$\delta_1 \geq 2^{-11r^3},$$
$$c \geq 2^{-12r^4}.$$

When we constructed the upper bound via product colorings, there was a weighted graph (namely the one used to construct the coloring) so that for any $S \subseteq R$ we could approximate the size of the largest clique using colors from $S$ by the product of the weights of edges contained in $S$. We wish to say that the structure of any Gallai coloring $F$ can be approximated this way. Though this is not true in general, the next theorem states that if it is not true then $\prod_S \g{F}{S}$ must be large. Take for the rest of this paper 
\begin{equation}\label{definitionm}
m_0 := 2^{2^{2^{2^{8r^2}}}}.
\end{equation}

\begin{theorem}
If $m \geq m_0$ then for any Gallai coloring $F$ on $n \leq m$ vertices, there are $f \geq 1$, $\epsilon \geq 0$, $\mathcal{P} \subseteq {R \choose 2}$, and, for $P \in \mathcal{P}$, weights $w_P \in [1,\infty)$ satisfying:
\begin{enumerate}
\item For every $S \in {R \choose s}$, $\g{F}{S} \geq \prod_{P \in {S \choose 2} \cap \mathcal{P}} w_P.$
\item $\prod_{P \in \mathcal{P}}w_P \geq m^{-\epsilon}n$.
\item $\prod_{S \in {R \choose s}} \g{F}{S} \geq (nf)^{r-2 \choose s-2}.$
\item $f \geq (\log m)^{C \epsilon}$.
\item Taking $a$ to be the size of $\mathcal{P}$, $f \geq \left(c\log^2 m\right)^{ad}$.
\end{enumerate}
\end{theorem}

From the above theorem we will quickly be able to conclude Theorem \ref{firsttheorem}. Note that if $f$ is large enough then by condition (3) we conclude that $\prod_{S \in {R \choose s}} \g{F}{S}$ is large and so some $\g{F}{S}$ is large. Otherwise, by condition (4) we have an upper bound on the size of $\epsilon$, so by conditions (1) and (2) the structure of the coloring is well-approximated by the $w_P$. This latter case will allow us to apply our work on weighted graphs from the previous subsection to get a lower bound on $a$, and then we will apply condition (5) to as before conclude that some $\g{F}{S}$ is large.

\begin{proof}
We will write $\gc{S}$ for $\g{F}{S}$. We will take $w_P = 1$ for any $P$ in ${R \choose 2}$ but not in $\mathcal{P}$; this way, for any $T \subseteq {R \choose 2}$, we have $\prod_{P \in T \cap \mathcal{P}}w_P = \prod_{P \in T} w_P$.

We proceed by induction on $n$.

\noindent {\bf Base Case:}
If $n=1$, then we may take $f=1,\epsilon=0$, and $\mathcal{P} = \emptyset$. Letting $a = \size{P} = 0$,

\begin{enumerate}
\item For every $S \in {R \choose s}$, $\gc{S} = 1 = \prod_{P \in {S \choose 2}}w_P$.
\item $\prod_{P \in {R \choose 2}}w_P = 1 = m^{-\epsilon}n$.
\item $\prod_{S \in {R \choose s}} \gc{S} = 1 = (nf)^{r-2 \choose s-2}$.
\item $f = 1 = (\log m)^{C \epsilon}$.
\item $f = 1 = \left(c \log^2m\right)^{ad}$.
\end{enumerate}

\noindent {\bf Preliminary Discussion:}
If $n>1$, there is some pair of colors $Q = \{Q_1,Q_2\}$ and there is a non-trivial partition $V(K_n)=V_1 \cup \ldots \cup V_t$ 
with $|V_1| \geq \ldots \geq |V_t|$ such that there is some $2$-coloring $\chi : {t \choose 2} \rightarrow Q$ such that for every distinct $i,j \in [t]$ and $u \in V_i$, $v \in V_j$, the color under $F$ of $\{u,v\}$ is $\chi(i,j)$ (which is in $Q$).

Given $\epsilon > 0$, define $f_\epsilon(\ell) := (\log m)^{C\frac{\log(\alpha m^{\epsilon})}{\log m}}$, where $\alpha = \ell/n$. Note that we may rewrite $f_\epsilon(\ell)=(\log m)^{C\epsilon + C\frac{\log\alpha}{\log m}}$; we will move between the two expressions freely. Note also that $f_\epsilon(\ell)$ is an increasing function of $\ell$. We will need some lemmas about $f_\epsilon$, all of which are formalizations of the statement ``$f_\epsilon$ does not grow too quickly".

\begin{lemma}\label{factsaboutfepsilon}
The following statements hold about $f_\epsilon$ for every choice of $\epsilon \geq 0$, $m \geq m_0$, and $1 < n \leq m$.
\begin{enumerate}
\item For any $\alpha\in [\frac{1}{n},1]$, 
$$f_\epsilon(\alpha n) \geq \alpha^{1/\left(2 {r-2 \choose s-2}\right)} f_\epsilon(n).$$ 
In particular, $f_\epsilon(\alpha n) \geq \alpha f_\epsilon(n)$.
\item For any $\alpha_1,\alpha_2,\alpha_3 \in [\frac{1}{n},1]$ with $\alpha_1+\alpha_2+\alpha_3=1$, taking $n_i = \alpha_i n$,
$$nf_\epsilon(n) \leq \sum_i n_if_\epsilon(n_i) + 3(\log^{-3/4}m)nf_\epsilon(n).$$
\item For $i \geq 0$ and $m^{\delta} \geq 2^j \geq 1$ we have $f_\epsilon(2^i)\log^{2/{r-2 \choose s-2}}((\log^{1/4}m)2^j) \geq 256{r \choose 2} f_\epsilon(2^{i+2j})$.
\item For any $\alpha \geq \log^{-1}m$, $f_\epsilon(\alpha n) \geq f_\epsilon(n)/2$.
\end{enumerate}
\end{lemma}

We will refer to the above collectively as the facts about $f_\epsilon$; we prove them in Appendix \ref{appendixC}.

The proof will split into four cases.

\noindent {\bf Cases 1 and 2, Preliminary Discussion:}
For these cases, a simple numerical claim will be useful.

\begin{claim}\label{exponentiationlemma}
For positive reals $a,b$ with $a \leq 1$,
$$(1+a)^b \geq 1+ab/2.$$
\end{claim}
\begin{proof}
Since $0 \leq a \leq 1$ we have $1+a \geq e^{a/2}$. Then
$$(1+a)^b \geq e^{ab/2} \geq 1+ab/2.$$
\end{proof}

Cases 1 and 2 will be those cases in which $V_1$ is large.

Take $U_1 = V_1$. Take $U_2$ to be the union of the $V_j$ such that the edges between $V_1$ and $V_j$ are of color $Q_1$. Take $U_3$ to be the union of the $V_j$ such that the edges between $V_1$ and $V_j$ are of color $Q_2$. We may assume without loss of generality that $\size{U_2} \geq \size{U_3}$. Then $\{U_1,U_2,U_3\}$ is a partition of $V$.

Define, for $i =1,2,3$, $F_i$ to be the restriction of $F$ to $U_i$. Take $\g{i}{S}:=\g{F_i}{S}$. Define $n_i = \size{U_i}$ and $\alpha_i = n_i/n$. By the induction hypothesis, for each $F_i$ there are appropriate choices of $f_i,\epsilon_i,\mathcal{P}_i,$ and $w_{P,i}$. Take $a_i = \size{\mathcal{P}_i}$.

The general approach for these cases as well as for Case 3 will be to choose some index $i$ and simply use the same graph to approximate our coloring. That is, we will take $\mathcal{P} = \mathcal{P}_i$ and $w_P = w_{P,i}$, and then we will show that $\epsilon$ and $f$ may be chosen appropriately.

We now proceed: if for some index $i$ we take $\mathcal{P} = \mathcal{P}_i$ and $w_P = w_{P,i}$, since $\gc{S} \geq \g{i}{S}$ we will have property (1): for every $S \in {R \choose s}$, $\gc{S} \geq \g{i}{S} \geq \prod_{P \subseteq S} w_{P,i} = \prod_{P \subseteq S} w_P$. Furthermore,

$$\prod_{P \in {R \choose 2}} w_P = \prod_{P \in {R \choose 2}} w_{P,i} \geq m^{-\epsilon_i}n_i = m^{-\epsilon_i}\alpha_in = m^{-\epsilon_i - \log(1/\alpha_i)/\log m}n.$$

If we take $\epsilon = \epsilon_i + \frac{\log(1/\alpha_i)}{\log m}$, then the above shows property (2) will hold.

Define
$$x_i := \max\left((\log m)^{C\left(\epsilon_i + \frac{\log(1/\alpha_i)}{\log m}\right)},(c\log^2m)^{a_id}\right).$$

If $i$ is the index minimizing $x_i$, then we will take:
$$\epsilon = \epsilon_i + \frac{\log(1/\alpha_i)}{\log m},$$
$$\mathcal{P}=\mathcal{P}_i,$$
$$w_P = w_{P,i},$$
and $f = x_i$; we will show that this satisfies properties (4) and (5), so choosing $i$ to minimize $x_i$ minimizes our $f$. Take $a = \size{\mathcal{P}}$. We have already observed that properties (1) and (2) will hold.

Take $\epsilon' = \max\left(\epsilon,\frac{\log((c\log^2 m)^{ad})}{C\log\log m}\right)$. Note that $f = x_i = (\log m)^{C \epsilon'} = f_{\epsilon'}(n)$. In this case properties (4) and (5) hold by choice of $f$:
$$f \geq (\log m)^{C \epsilon'} \geq (\log m)^{C \epsilon}.$$
$$f \geq (\log m)^{C \epsilon'} \geq (\log m)^{C\frac{\log((c\log^2 m)^{ad})}{C\log\log m}} = 2^{\log((c\log^2 m)^{ad})} = (c\log^2 m)^{ad}.$$

We have only to show that, with this choice of $f$, property (3) holds. We claim that each $f_i$ satisfies $f_i \geq f_{\epsilon'}(n_i)$.

If for some $i$ we have $\epsilon_i < \epsilon' + \frac{\log\alpha_i}{\log m}$, then we must have $(c \log^2m)^{a_id} \geq f = (\log m)^{C\epsilon'}$, for otherwise we would have $x_i < f$, contradicting our choice of $\epsilon$. Therefore, for such an index $i$,
$$f_i \geq (c\log^2m)^{a_id} \geq (\log m)^{C\epsilon'} = f_{\epsilon'}(n) \geq f_{\epsilon'}(n_i).$$
Otherwise, $\epsilon_i \geq \epsilon' + \frac{\log \alpha_i}{\log m}$ so
$$f_i \geq (\log m)^{C \epsilon_i} \geq (\log m)^{C\left(\epsilon' + \frac{\log\alpha_i}{\log m}\right)} = f_{\epsilon'}(n_i).$$

We have that for each $S$ satisfying $Q \subseteq S$ that $\gc{S} \geq \g{1}{S}+\g{2}{S}+\g{3}{S}$. For each $S$ satisfying $Q_1 \in S$, we have 
$$\gc{S} \geq \max(\g{1}{S}+\g{2}{S},\g{3}{S}) \geq \g{1}{S}+\g{2}{S}.$$ Similarly, if $Q_2 \in S$ then $\gc{S} \geq \g{1}{S}+\g{3}{S}.$ Finally, for all $S$ we have $\gc{S} \geq \max_i \g{i}{S}$.

We have by the generalization of H\"older's inequality (Lemma \ref{holder}):

$$\prod_S \gc{S} \geq \prod_{S : Q \subseteq S} \sum_i \g{i}{S} \prod_{S : Q \not \subseteq S} \gc{S} \geq \left(\sum_i \left(\prod_{S : Q \subseteq S}\g{i}{S} \prod_{S : Q \not \subseteq S} \gc{S} \right)^{1/{r-2 \choose s-2}}\right)^{r-2 \choose s-2}.$$

Therefore, we need only check that

$$\sum_i \left(\prod_{S : Q \subseteq S}\g{i}{S} \prod_{S : Q \not \subseteq S} \gc{S} \right)^{1/{r-2 \choose s-2}} \geq nf.$$

Fix $T \in {R \choose s}$ so that $T \cap Q = \{Q_1\}$. We get $\gc{T} \geq \g{1}{T} + \g{2}{T}$.

\noindent {\bf Case 1:} $\alpha_1, \alpha_2 \geq \log^{-1/2}m$.
The argument from the weak lower bound case applied here only gives:
$$\sum_i \left(\prod_{S : Q \subseteq S}\g{i}{S} \prod_{S : Q \not \subseteq S} \gc{S} \right)^{1/{r-2 \choose s-2}} \geq \sum_i n_if_{\epsilon'}(n_i).$$
The main idea behind solving this case is to observe that it is sufficient to gain a constant factor on either the largest or second largest term of the above sum.

Consider:
$$\sum_i \left(\prod_{S : Q \subseteq S} \g{i}{S} \prod_{S : Q \not \subseteq S} \gc{S}\right)^{1/{r-2 \choose s-2}} \geq$$ 
$$\left((\g{1}{T}+\g{2}{T})\prod_{S \neq T} \g{1}{S} \right)^{1/{r-2 \choose s-2}} + \left((\g{1}{T}+\g{2}{T})\prod_{S \neq T} \g{2}{S} \right)^{1/{r-2 \choose s-2}} + \left(\prod_S \g{3}{S}\right)^{1/{r-2 \choose s-2}}.$$

We'll handle the case $\g{1}{T} \leq \g{2}{T}$; the case where $\g{1}{T} \geq \g{2}{T}$ has a symmetric argument. Then, since $\g{1}{T} + \g{2}{T} \geq 2 \g{1}{T}$, the previous is at least:

\begin{eqnarray*}
& & 2^{1/{r-2 \choose s-2}}\prod_S  \g{1}{S}^{1/{r-2 \choose s-2}} + \prod_S \g{2}{S}^{1/{r-2 \choose s-2}} + \prod_S \g{3}{S}^{1/{r-2 \choose s-2}} 
\\ & = & \sum_i \prod_S \g{i}{S}^{1/{r-2 \choose s-2}} + \left(2^{1/{r-2 \choose s-2}}-1\right)\prod_S \g{1}{S}^{1/{r-2 \choose s-2}} \\ 
& \geq & \sum_i n_if_i + ((1+1)^{1/{r-2 \choose s-2}}-1)n_1f_1 \\
& \geq & \sum_i n_if_{\epsilon'}(n_i) + \left(2{r-2 \choose s-2}\right)^{-1}n_1f_{\epsilon'}(n_1) \\ 
& \geq & \sum_i n_if_{\epsilon'}(n_i) + \left(2{r-2 \choose s-2}\right)^{-1}(\log^{-1/2}m)nf_{\epsilon'}((\log^{-1/2}m)n) \\
& \geq & \sum_i n_if_{\epsilon'}(n_i) + \left(4{r-2 \choose s-2}\right)^{-1}(\log^{-1/2}m)nf_{\epsilon'}(n) \\
& \geq & \sum_i n_if_{\epsilon'}(n_i) + 3(\log^{-3/4}m)nf_{\epsilon'}(n) \geq nf_{\epsilon'}(n) = nf,
\end{eqnarray*}
where the first follows from the induction hypothesis, the second follows from Claim \ref{exponentiationlemma}, the third follows from the lower bound on $\alpha_1$, the fourth follows from the fourth fact about $f_{\epsilon'}$, the fifth follows from $m \geq m_0$, and the sixth follows from the second fact about $f_{\epsilon'}$.

\noindent {\bf Case 2:} $\log^{-1/2}m \geq \alpha_2$.
In this case we have $\alpha_1 = 1-(\alpha_2+\alpha_3) \geq 1-2\alpha_2 \geq 1-2\log^{-1/2}m \geq 3/4$.

Again, the argument from the weak lower bound only gives:
$$\sum_i \left(\prod_{S : Q \subseteq S}\g{i}{S} \prod_{S : Q \not \subseteq S} \gc{S} \right)^{1/{r-2 \choose s-2}} \geq \sum_i n_if_{\epsilon'}(n_i).$$
The main idea behind this case is to observe that it is sufficient to gain either a factor of $(1+8\alpha_2)$ on the first term (which is much larger than the others) or a factor of $4 \alpha_2^{-1/\left(2{r -2 \choose s-2}\right)}$ on the second term. We will do the former if $\g{2}{T}/\g{1}{T}$ is large enough, and otherwise we may accomplish the latter.

If $\g{2}{T} \geq 16{r-2 \choose s-2}\alpha_2\g{1}{T}$, we have

$$\sum_i\left(\prod_{S : Q \subseteq S} \g{i}{S}\prod_{S : Q \not \subseteq S} \gc{S}\right)^{1/{r-2 \choose s-2}} \geq \left(\prod_{S : Q \subseteq S} \g{1}{S}\prod_{S : Q \not \subseteq S} \gc{S}\right)^{1/{r-2 \choose s-2}} \geq$$ 
$$(\g{1}{T} + \g{2}{T})^{1/{r-2 \choose s-2}}\prod_{S \neq T} \g{1}{S}^{1/{r-2 \choose s-2}} \geq \left(1+16{r-2 \choose s-2}\alpha_2\right)^{1/{r-2 \choose s-2}}\prod_S \g{1}{S}^{1/{r-2 \choose s-2}} \geq$$ 
$$\left(1+\frac{16{r-2 \choose s-2}\alpha_2}{2 {r-2 \choose s-2}}\right)n_1f_1 = n_1f_1 + 8 \alpha_2 n_1f_1,$$
where the last inequality is by Claim \ref{exponentiationlemma}.

We know $f_i \geq f_{\epsilon'}(n_i)$, so the above is at least:
$$n_1f_{\epsilon'}(n_1) + 8 \alpha_2 n_1f_{\epsilon'}(n_1) \geq \alpha_1^2nf_{\epsilon'}(n) + 8 \alpha_2 \alpha_1^2 nf_{\epsilon'}(n) \geq$$ 
$$(1-2\alpha_2)^2nf_{\epsilon'}(n) + 4\alpha_2 nf_{\epsilon'}(n) > nf_{\epsilon'}(n) = nf,$$
where the first inequality follows from the first fact about $f_{\epsilon'}$ and the second inequality from substituting lower bounds on $\alpha_1$.

Otherwise, we have $\g{2}{T} \leq 16{r-2 \choose s-2} \alpha_2 \g{1}{T}$, so

$$\sum_i\left(\prod_{S : Q \subseteq S} \g{i}{S}\prod_{S : Q \not \subseteq S} \gc{S}\right)^{1/{r-2 \choose s-2}} \geq \prod_S \g{1}{S}^{1/{r-2 \choose s-2}} + \left((\g{1}{T} + \g{2}{T})\prod_{S \neq T} \g{2}{S}\right)^{1/{r-2 \choose s-2}}.$$

Then the latter term is at least:

$$\left(\g{1}{T} \prod_{S \neq T} \g{2}{S}\right)^{1/{r-2 \choose s-2}} \geq \left(\frac{1}{16{r-2 \choose s-2} \alpha_2} \prod_S \g{2}{S}\right)^{1/{r-2 \choose s-2}} \geq$$
$$\left(\frac{1}{16{r-2 \choose s-2} \alpha_2}\right)^{1/{r-2 \choose s-2}}n_2f_2 \geq 4\alpha_2^{-1/\left(2{r-2 \choose s-2}\right)}n_2f_2 \geq 4\alpha_2^{-1/\left(2{r-2 \choose s-2}\right)}n_2f_{\epsilon'}(n_2) \geq $$
$$4\alpha_2^{-1/\left(2{r-2 \choose s-2}\right)}n_2 \left(\alpha_2^{1/\left(2{r-2 \choose s-2}\right)} f_{\epsilon'}(n)\right) = 4 n_2f_{\epsilon'}(n),$$
where the third inequality follows from the upper bound on $\alpha_2$ and from $m \geq m_0$ and the fifth inequality from the first fact about $f_{\epsilon'}$.

Therefore,
$$\sum_i\left(\prod_{S : Q \subseteq S} \g{i}{S}\prod_{S : Q \not \subseteq S} \gc{S}\right)^{1/{r-2 \choose s-2}} \geq \prod_S \g{1}{S}^{1/{r-2 \choose s-2}} + 4n_2f_{\epsilon'}(n) \geq n_1f_{\epsilon'}(n_1) + 4n_2f_{\epsilon'}(n) \geq$$
$$\alpha_1^2 nf_{\epsilon'}(n) + 4\alpha_2 n f_{\epsilon'}(n) \geq (1-2 \alpha_2)^2 nf_{\epsilon'}(n) + 4 \alpha_2 n f_{\epsilon'}(n) \geq n f_{\epsilon'}(n) = nf,$$
where the third inequality follows from the first fact about $f_{\epsilon'}$.

\noindent{\bf Cases 3 and 4, Preliminary Discussion:} 
These will be the cases in which none of the $V_i$ are large. For these cases, we will take $n_i = \size{V_i}$ and $\alpha_i = n_i/n$. We will take $F_i$ to be $F$ restricted to $V_i$ and take $\g{i}{S} = \g{F_i}{S}$. By induction, for each $F_i$ there are appropriate choices of $f_i,\epsilon_i, \mathcal{P}_i,$ and $w_{P,i}$. Take $a_i = \size{\mathcal{P}_i}$.

Since in these cases we have many indices, we will be able to apply the weighted Ramsey's theorem to appropriately selected subsets of them. The rest of the preliminary discussion for Cases 3 and 4 is based on doing so.

For each non-negative integer $i$ let $I_i :=[2^i,2^{i+1}]$. Take $\Phi(i) = \{j : n_j \in I_i\}$. The $\Phi(i)$ form a dyadic partition of the indices which will eventually determine how the indices are clustered when we apply the weighted Ramsey's theorem. Take $B'' := \{i \leq \log(nm^{-\delta})\}$.

Take $\tau = \lfloor \log(2(\log^{-1/2}m)n) \rfloor$, so $\max_i n_i' \leq (\log^{-1/2}m)n \leq 2^\tau \leq 2(\log^{-1/2}m)n$ and $\Phi(i)$ is empty for $i > \tau$.

For any pair $T,T'$ satisfying $Q \cap T = \{Q_1\}$ and $Q \cap T' = \{Q_2\}$, take $g_i=\prod_{S \not \in \{T,T'\}}\g{i}{S},o_i = \g{i}{T},p_i=\g{i}{T'}$ and $g=\prod_{S \not \in \{T,T'\}} \gc{S},o=\gc{T},p=\gc{T'}$. Take $G_{T,T'}$ to be the set of indices $j$ with $op \geq o_jp_j \log^2((\log^{1/4}m)2^{(\tau-i)/2})/32$ where $i$ is such that $j \in \Phi(i)$.

$G_{T,T'}$ is the collection of indices $j$ for which $\gc{T}\gc{T'}$ is substantially larger than $\g{j}{T}\g{j}{T'}$, where the meaning of ``substantially larger" depends on the size of $n_j$. The following lemma states that almost all vertices are contained in some $V_j$ as $j$ varies over the indices of $G_{T,T'}$.

\begin{claim}
$\sum_{j \in G_{T,T'}} n_j \geq (1-\delta_1)n$.
\end{claim}
\begin{proof}
Take $\{V_i'\}_{i \leq t}$ to be a reordering of $\{V_i\}_{i \leq t}$ so that if $i \leq j$ then $\g{V_i'}{T}\g{V_i'}{T'} \leq \g{V_j'}{T}\g{V_j'}{T'}$. That is, the $V_i'$ are in increasing order based on the value of $\g{V_i'}{T}\g{V_i'}{T'}$. Let $g_j',o_j',p_j',\Phi(i)',n_j',G_{T,T'}'$ be defined as before (so, for example, $o_i' = \g{V_i'}{T}$).

When we count, we wish to omit certain intervals that do not satisfy desired properties. Let
$$B:=\{i \leq \tau : \size{\Phi(i)'} \leq 4\delta_1^{-1} (\log^{1/4}m) 2^{(\tau-i)/2}\}.$$
We will show that a large fraction of the vertices are not contained in $\bigcup_{i \in B} \cup_{j \in \Phi(i)'} V_j'$; that is, most vertices are not contained in $V_j'$ where $j$ is an index in $\phi(i)'$ for some $i \in B$.

$$\sum_{i \in B} \sum_{j \in \Phi(i)'} n_j' \leq \sum_{i \leq \tau} 2^{i+1}(4 \delta_1^{-1}(\log^{1/4}m) 2^{(\tau-i)/2}) = 8\delta_1^{-1}(\log^{1/4}m)2^{\tau/2}\sum_{i \leq \tau} 2^{i/2} \leq $$
$$8\delta_1^{-1}(\log^{1/4}m)2^{\tau/2}4 \cdot 2^{\tau/2} = 32\delta_1^{-1}(\log^{1/4}m)2^\tau \leq 64\delta_1^{-1} (\log^{-1/4}m)n \leq \delta_1 n/2.$$

Thus,

\begin{equation}
\label{fewbad}
\sum_{i \in B} \sum_{j \in \Phi(i)'}n_j' \leq \delta_1n/2
\end{equation}

For any fixed $i \leq \tau$ such that $i \not \in B$, enumerate $\Phi(i)'$ as $\phi_{i,1},\phi_{i,2},\ldots,\phi_{i,\size{\Phi(i)'}}$ with $o_{\phi_{i,j_1}}'p_{\phi_{i,j_1}}' \leq o_{\phi_{i,j_2}}'p_{\phi_{i,j_2}}'$ if $j_1 \leq j_2$. That is, this enumeration is so that the $V_{\phi_{i,j}}'$ are listed in increasing order with respect to their $o_j'p_j'$ values. Take $\beta_i$ to be $\phi_{i,\left(1-\delta_1/4\right)\size{\Phi(i)'}}$. Consider $\{(o_j',p_j') : j \in \Phi(i)', j \geq \beta_i\}$. By $i \not \in B$, this has at least $(\log^{1/4}m)2^{(\tau-i)/2} \geq M$ elements. We get by applying the weighted Ramsey's theorem to this set (with the coloring given by $\chi$) that:
$$op \geq o_{\beta_i}'p_{\beta_i}'\frac{\log^2\left((\log^{1/4}m)2^{(\tau-i)/2}\right)}{32}.$$
For any $j \in \Phi(i)'$ we have that if $j \leq \beta_i$ then $o_j'p_j' \leq o_{\beta_i}'p_{\beta_i}'$, so the above is at least $o_j'p_j' \frac{\log^2\left((\log^{1/4}m)2^{(\tau-i)/2}\right)}{32},$ so $j \in G_{T,T'}'$.

If $i \not \in B$ we get:

$$\sum_{j \in \Phi(i)' : j \not \in G_{T,T'}'} n_j' \leq \sum_{j \in \Phi(i)' : j > \beta_i} n_j' \leq \frac{\delta_1}{4}\size{\Phi(i)'}2^{i+1} =$$ 
$$\size{\Phi(i)'}2^i\delta_1/2 \leq \frac{\delta_1}{2}\sum_{j \in \Phi(i)'}n_j'.$$
Therefore,
$$\sum_{j \in \Phi(i)' \cap G_{T,T'}'} n_j' \geq \left(1-\delta_1/2\right)\sum_{j \in \Phi(i)'}n_j'.$$

Thus,
$$\sum_{j \in G_{T,T'}'} n_j' \geq \sum_{i \not \in B} \sum_{j \in \Phi(i)' \cap G_{T,T'}'} n_j' \geq \left(1-\delta_1/2\right) \sum_{i \not \in B} \sum_{j \in \Phi(i)'}n_j' \geq$$
$$\left(1-\delta_1/2\right)^2 n \geq (1-\delta_1)n,$$
where the third inequality follows from \eqref{fewbad}.

As $\sum_{j \in G_{T,T'}}n_j' = \sum_{j \in G_{T,T'}}n_j$, this completes the proof of the claim.

\end{proof}

\noindent{\bf Case 3:} $\alpha_1 \leq \log^{-1/2} m$ and $\sum_{i \in B''}\sum_{j \in \Phi(i)} n_j \leq n/2$.
Fix any pair $T,T' \in {R \choose s}$ with $T \cap Q = Q_1$ and $T' \cap Q = Q_2$. The idea behind this case will be to choose some subset $G_a$ of the indices so that as $j$ varies over $G_a$ the value of $a_j$ doesn't change, to intersect this set with $G_{T,T'}$, and to use this to show that $\prod_S \gc{S}$ is large. In this case, as in Cases 1 and 2, we will simply choose some $j$ appropriately and take $\mathcal{P} = \mathcal{P}_j$ and $w_P = w_{P,j}$.

Note that $s_i \in \left[{r \choose 2}\right]$, so by the pigeonhole principle there must be some value $s$ so that 
$$\sum_{i \not \in B''} \sum_{j \in \Phi(i) : s_i = s}n_j \geq {r \choose 2}^{-1} \sum_{i \not \in B''} \sum_{j \in \Phi(i)} n_j \geq {r \choose 2}^{-1}n/2,$$
where the last inequality follows by the assumptions for Case 3. Take $G_a$ to be the set of indices $j$ with $a_j=a$ and, taking $i$ to be the index with $j \in \Phi(i)$, $i$ is not in $B''$; by the above, $\sum_{j \in G_a} n_j \geq {r \choose 2}^{-1}n/2$. Then take $\epsilon = \min_{i \in G_a} \epsilon_i + \frac{\log(1/\alpha_i)}{\log m}$. Note that this is the same as taking $\epsilon = \epsilon_i + \frac{\log(1/\alpha_i)}{\log m}$ where $i$ is an index in $G_a$ minimizing
$$x_i := \max\left((\log m)^{C\left(\epsilon_i + \frac{\log(1/\alpha_i)}{\log m}\right)},(c\log^2m)^{a_id}\right),$$
as all the $a_i$ are equal to $a$.

Take, with $i$ as above, $\mathcal{P}=\mathcal{P}_i$ and $w_P = w_{P,i}$. Then, as in Cases 1 and 2, properties (1) and (2) hold. Furthermore, as in Cases 1 and 2, taking 
$$\epsilon' = \max\left(\epsilon,\frac{\log\left((c\log^2m)^{ad}\right)}{C\log\log m}\right),$$
we have for $i \in G_a$ that $f_i \geq f_{\epsilon'}(n_i)$, and taking $f = f_{\epsilon'}(n)$, properties (4) and (5) hold. We need only check that property (3) holds.

Then take $G = G_a \cap G_{T,T'}$.

We have 
\begin{equation} \label{goodbig3}\sum_{j \in G} n_j \geq n - \sum_{j \not \in G_a} n_j - \sum_{j \not \in G_{T,T'}} n_j \geq \left({r \choose 2}^{-1}/2 - \delta_1\right)n \geq \left({r \choose 2}^{-1}/4\right)n.\end{equation}
Take $g_i=\prod_{S \not \in \{T,T'\}}\g{i}{S},o_i = \g{i}{T},p_i=\g{i}{T'}$ and $g=\prod_{S \not \in \{T,T'\}} \gc{S},o=\gc{T},p=\gc{T'}$.

We have:
\begin{eqnarray*}
\sum_j (g_jop)^{1/{r-2 \choose s-2}} & \geq & \sum_i \sum_{j \in G \cap \Phi(i)} (g_jop)^{1/{r-2 \choose s-2}} \\
& \geq & \sum_i\sum_{j \in G \cap \Phi(i)} \left( g_jo_jp_j \frac{\log^2((\log^{1/4}m)2^{(\tau-i)/2})}{32} \right)^{1/{r-2 \choose s-2}} \\
& \geq & \sum_i \sum_{j \in G \cap \Phi(i)}n_jf_j\log^{2/{r-2 \choose s-2}}((\log^{1/4}m)2^{(\tau-i)/2})32^{-1/{r-2 \choose s-2}}\\
& \geq & \sum_i \sum_{j \in G \cap \Phi(i)}n_jf_{\epsilon'}(n_j)\log^{2/{r-2 \choose s-2}}((\log^{1/4}m)2^{(\tau-i)/2})32^{-1/{r-2 \choose s-2}} \\
& \geq & \sum_i \sum_{j \in G \cap \Phi(i)}n_jf_{\epsilon'}(2^i)\log^{2/{r-2 \choose s-2}}((\log^{1/4}m)2^{(\tau-i)/2})32^{-1} \\
& \geq & \sum_i \sum_{j \in G \cap \Phi(i)} 8{r \choose 2}n_j f_{\epsilon'}(2^\tau) \\
& \geq & \sum_i \sum_{j \in G \cap \Phi(i)} 8{r \choose 2}n_j f_{\epsilon'}(n\log^{-1/2}m) \\
& \geq & \sum_i \sum_{j \in G \cap \Phi(i)} 4{r \choose 2}n_j f_{\epsilon'}(n) \geq nf_{\epsilon'}(n) = nf,
\end{eqnarray*}
where the second inequality follows from $G \subseteq G_{T,T'}$, the sixth inequality is by the third fact about $f_{\epsilon'}$, the eighth inequality is by the fourth fact about $f_{\epsilon'}$, and the ninth inequality inequality follows from \eqref{goodbig3}.

Note that here we used only one pair $T,T'$; we can afford to do this because we gain a large amount due to $B''$ being small. In the next case, we will use all of the relevant pairs.

\noindent{\bf Case 4:} $\alpha_1 \leq \log^{-1/2}m$ and $\sum_{i \in B''} \sum_{j \in \Phi(i)} n_j \geq n/2$. In this case there are many vertices contained in very small parts; this is the case where we will not simply take $\mathcal{P}$ to be some $\mathcal{P}_i$.

The idea behind this case is to choose a set $G_a$ of many indices $j$ with similar values of $\mathcal{P}_j,w_{P,j}$, and $\epsilon_j$. We will be able to take $w_Q = \sum_{j \in G_a} w_{Q,j}$, which is a significant improvement over simply taking for some $j$ each $w_P = w_{P,j}$. We will intersect $G_a$ with some collection of $G_{T,T'}$ where each $T$ with $T \cap Q = \{Q_1\}$ and each $T'$ with $T' \cap Q = \{Q_2\}$ will occur exactly once (so we pair up the sets $T,T'$ with $T \cap Q = \{Q_1\}$ and $T' \cap Q = \{Q_2\}$; if an index is in $G_{T,T'}$ then we have gained a large factor on that index). This allows us to lower bound $\prod_S \gc{S}$.

We may partition $[0,1]$ into at most $\delta_0^{-1}+1$ intervals $J_i$ of length at most $\delta_0$. Furthermore, we may partition $[1,m]$ into at most $\delta_0^{-1}+1$ intervals $H_i$ with $\sup(H_i)/\inf(H_i) \leq m^{\delta_0}$.

We partition the indices in $\bigcup_{i \in B''}\Phi(i)$ by saying two indices $j,j'$ are in the same part if and only if $\epsilon_j,\epsilon_{j'}$ are in the same interval $J_i$, $\mathcal{P}_j=\mathcal{P}_{j'}$, and for each $P \in \mathcal{P}_j=\mathcal{P}_{j'}$, $w_{P,j}$ and $w_{P,j'}$ are in the same interval $H_i$.

Then the total number of possible partitions is at most $2^{r \choose 2}(\delta_0^{-1}+1)^{{r \choose 2}+1}$. Therefore, there is some part $G_a \subseteq \bigcup_{i \in B''}\Phi(i)$ where 
$$\sum_{j \in G_a} n_j \geq 2^{-{r \choose 2}}(\delta_0^{-1}+1)^{-{r \choose 2}-1} \sum_{i \in B''} \sum_{j \in \Phi(i)} n_j \geq 2^{-{r \choose 2}-1}(\delta_0^{-1}+1)^{-{r \choose 2}-1}n = 2{r-2 \choose s-1}\delta_1 n.$$

Then take $\epsilon_0 = \max_{i \in G_a} \epsilon_i$. Take $w_Q = \sum_{i \in G_a} w_{Q,i}$ and for $P \neq Q$ take $w_P = \min_{i \in G_a} w_{P,i}$. Take $\mathcal{P} = \mathcal{P}_i \cup \{Q\}$ for any $i \in G_a$. Take $a = \size{\mathcal{P}}$ and note $a \leq a_i + 1$ for any $i \in G_a$. We check that property (1) holds. For each $S$ with $Q \subseteq S$, 

$$\gc{S} \geq \sum_i \g{i}{S} \geq \sum_i \prod_{P \subseteq S} w_{P,i} \geq \sum_{i \in G_a} w_{Q,i}\prod_{P \subseteq S, P \neq Q}\min_{i \in G_a} w_{P,i} = w_Q \prod_{P \subseteq S, P \neq Q} w_P = \prod_{P \subseteq S} w_P.$$

If $Q \not \subseteq S$ then, fixing any $i \in G_a$,

$$\gc{S} \geq \g{i}{S} \geq \prod_{P \subseteq S} w_{P,i} \geq \prod_{P \subseteq S} \min_{i \in G} w_{P,i} = \prod_{P \subseteq S} w_P.$$

Now, we choose $\epsilon = {r \choose 2}\delta_0 + \epsilon_0$ and check that property (2) holds:

$$\prod_P w_P = \sum_{i \in G_a} w_{Q,i} \prod_{P \neq Q} w_P \geq \sum_{i \in G_a} w_{Q,i} \prod_{P \neq Q} (m^{-\delta_0}w_{P,i}) =$$ 
$$m^{-\left({r \choose 2}-1\right)\delta_0}\sum_{i \in G_a} \prod_P w_{P,i} \geq m^{-\left({r \choose 2}-1\right)\delta_0-\epsilon_0}\sum_{i \in G_a} n_i \geq$$ 
$$m^{-\left({r \choose 2}-1\right)\delta_0 - \epsilon_0} 2 {r-2 \choose s-1} \delta_1 n \geq m^{-{r \choose 2}\delta_0-\epsilon_0}n,$$
where the first inequality is valid since for $j,j' \in G_s$ we have $w_{P,j}/w_{P,j'} \leq m^{\delta_0}$, the second inequality follows by choice of $\epsilon_0 = \max_{j \in G_a} \epsilon_j$, and the last inequality uses $m^{\delta_0} \geq 2{r-2 \choose s-1}\delta_1$ which follows from $m \geq m_0$ ($m_0$ is defined in Equation \ref{definitionm}) and the choices of $\delta_0$ and $\delta_1$.

Fix a bijection $\pi$ between $\{S \in {R \choose s} : S \cap Q = \{Q_1\}\}$ and $\{S \in {R \choose s} : S \cap Q = \{Q_2\}\}$ (one such bijection takes any $S$ in the first set and removes $Q_1$ and adds $Q_2$.) Take $G$ to be the intersection of $G_a$ and all sets of the form $G_{S,\pi(S)}$ where $S \in {R \choose s}$ satisfies $S \cap Q = \{Q_1\}$. There are ${r-2 \choose s-1}$ pairs $S,\pi(S)$, so by Claim \ref{fewbad} we have
$$\sum_{j \in G} n_j \geq \sum_{j \in G_a} n_j - \sum_{S : S \cap Q = \{Q_1\}}\sum_{j \not \in G_{S,\pi(S)}}n_j \geq \left(2{r-2 \choose s-1}\delta_1 - {r-2 \choose s-1}\delta_1\right) n \geq {r-2 \choose s-1} \delta_1n \geq \delta_1 n.$$

Note that if $j \in G$ then for any $S$ with $S \cap Q = \{Q_1\}$, since $G \subseteq G_{S,\pi(S)}$, we have $\gc{S}\gc{\pi(S)}\geq \g{j}{S}\g{j}{\pi(S)}\log^2((\log^{1/4}m)2^{(\tau-i)/2}) \geq \g{j}{S}\g{j}{\pi(S)}\log^2(2^{(\tau-i)/2})$ where $i$ is such that $j \in \Phi(i)$. However, if $j \in G$ then $i \in B''$, so

$$2^{(\tau-i)/2} = \left(\frac{2^\tau}{2^i}\right)^{1/2} \geq \left(\frac{n\log^{-1/2}m}{nm^{-\delta}}\right)^{1/2} \geq m^{\delta/4}.$$

Therefore, $\log(2^{(\tau-i)/2}) \geq \delta(\log m)/4$.

This gives:

$$\sum_j \left(\prod_{S : Q \subseteq S} \g{j}{S} \prod_{S : Q \not \subseteq S} \gc{S}\right)^{1/{r-2 \choose s-2}} \geq \sum_{j \in G} \left(\prod_{S : Q \subseteq S} \g{j}{S} \prod_{S : Q \not \subseteq S} \gc{S}\right)^{1/{r-2 \choose s-2}} \geq $$
$$\sum_{j \in G} \left(\prod_{S : \size{Q \cap S} \neq 1} \g{j}{S} \prod_{S : \size{Q \cap S}=1} \gc{S}\right)^{1/{r-2 \choose s-2}} = \sum_{j \in G} \left(\prod_{S : \size{Q \cap S} \neq 1} \g{j}{S} \prod_{S : Q \cap S = \{Q_1\}} \gc{S}\gc{\pi(S)}\right)^{1/{r-2 \choose s-2}} \geq $$
$$\sum_{j \in G} \left(\prod_{S : \size{Q \cap S} \neq 1} \g{j}{S} \prod_{S : Q \cap S = \{Q_1\}} \frac{\delta^2}{16}(\log^2m)\g{j}{S}\g{j}{\pi(S)}\right)^{1/{r-2 \choose s-2}} =
\sum_{j \in G} \left(\left(\delta(\log m)/16\right)^{2{r-2 \choose s-1}} \prod_S \g{j}{S}\right)^{1/{r-2 \choose s-2}} \geq$$
$$\sum_{j \in G} n_jf_j\left(\delta(\log m)/4\right)^{2{r-2 \choose s-1}/{r-2 \choose s-2}} = \sum_{j \in G} n_jf_j\left(\delta(\log m)/4\right)^{2d}.$$

Take $f' = (\log m)^{C \epsilon}$, $f'' = (c \log^2 m)^{ad}$ and $f = \max(f',f'')$. Note that $f \geq f'$ guarantees that property (4) holds and $f \geq f''$ guarantees that property (5) holds, so we need only check that property (3) holds. There will be two cases, that in which $f=f'$ and that in which $f=f''$.

If $f=f'$, for each $j \in G_a$ we have $\epsilon_j \geq \epsilon_0 - \delta_0$, so $f_j \geq (\log m)^{C(\epsilon_0-\delta_0)}$. Then we get:

$$\sum_{j \in G} n_jf_j\left(\delta(\log m)/4\right)^{2d} \geq \sum_{j \in G} n_j(\log m)^{C(\epsilon_0-\delta_0)}\left(\delta(\log m)/4\right)^{2d} \geq$$
$$\sum_{j \in G} n_j(\log m)^{C(\epsilon_0+{r \choose 2}\delta_0)}\left(\delta(\log^{1/2}m)/4\right)^{2d} \geq$$
$$\delta_1 n(\log m)^{C(\epsilon_0+{r \choose 2}\delta_0)}\left(\delta(\log^{1/2}m)/4\right)^{2d} \geq$$
$$n (\log m)^{C \left(\epsilon_0 + {r \choose 2}\delta_0\right)} = n (\log m)^{C \epsilon} = nf' = nf.$$

Otherwise, $f=f''$. For each $j \in G_a$ we have $f_j \geq (c \log^2m)^{(a-1)d}$, as $a \leq a_j+1$. This gives:

$$\sum_{j \in G} n_jf_j\left(\delta(\log m)/4\right)^{2d} \geq \sum_{j \in G} n_j(c \log^2m)^{(a-1)d}\left(\delta(\log m)/4\right)^{2d} = $$
$$(\delta/4)^{2d}c^{-d}(c \log^2m)^{ad} \sum_{j \in G} n_j \geq (\delta/4)^{2d}c^{-d}(c \log^2m)^{ad} \delta_1 n \geq$$
$$n (c \log^2m)^{ad} = nf''.$$

\end{proof}

Take $a_0$ to be ${r \choose 2}$ if $s < r-1$, $r/2$ if $s=r-1$ and $r$ is even, and $(r+3)/2$ if $s=r-1$ and $r$ is odd. Take $f_0 = (c\log^2m)^{a_0d}$. The following theorem states that either some $\gc{S}$ is large or their product is large.

\begin{theorem}
If $m \geq m_0$ either $\prod_S \gc{S} \geq (mf_0)^{r-2 \choose s-2}$ or there is some $S \subseteq R$ of size $s$ with $\gc{S} \geq (mf_0)^{{s \choose 2}/{r \choose 2}}$.
\end{theorem}
\begin{proof}
Choose $f,\epsilon,\mathcal{P},w_P$ as given by the previous theorem, then we need only show $f \geq f_0$. If $\epsilon \geq \left(16r {r \choose 2}^2\right)^{-1}$ then we have $C \epsilon \geq 2{r \choose 2}d \geq 2 a_0 d$ so $f \geq (\log m)^{C \epsilon} \geq (\log m)^{2 a_0 d} \geq f_0$.

Otherwise, $\epsilon < \left(16r {r \choose 2}^2\right)^{-1}$. Define a weighted graph on vertex set $R$ where an edge $e \in {R \choose 2}$ has weight $\log w_e$. Note this graph has non-negative edge weights and if an edge is not in $\mathcal{P}$, then it has weight $0$.

By Lemma \ref{weightedgraphs}, either this graph has at least $a_0$ edges or there is some set $S$ on $s$ vertices with 
$$\sum_{P \subseteq S} \log w_P \geq \left(1 + \left(4r {r \choose 2}^2\right)^{-1}\right)\frac{{s \choose 2}}{{r \choose 2}}\sum_P \log w_P.$$

If the graph has at least $a_0$ edges, then $\size{\mathcal{P}} \geq a_0$ so $f \geq (c \log^2m)^{ap_0 d}$, as desired.

Otherwise, there is some $S$ so that 
$$\sum_{P \subseteq S} \log w_P \geq \left(1 + \left(4r {r \choose 2}^2\right)^{-1}\right)\frac{{s \choose 2}}{{r \choose 2}}\sum_P \log w_P.$$

Then we have
$$\prod_{P \subseteq S} w_P \geq \prod_P w_P^{\left(1 + \left(4r {r \choose 2}^2\right)^{-1}\right)\frac{{s \choose 2}}{{r \choose 2}}} \geq m^{(1-\epsilon)\left(1 + \left(4r {r \choose 2}^2\right)^{-1}\right)\frac{{s \choose 2}}{{r \choose 2}}} \geq$$
$$m^{\left(1-\left(16r {r \choose 2}^2\right)^{-1}\right)\left(1 + \left(4r {r \choose 2}^2\right)^{-1}\right)\frac{{s \choose 2}}{{r \choose 2}}} \geq m^{\left(1 + \left(8r {r \choose 2}^2\right)^{-1}\right)\frac{{s \choose 2}}{{r \choose 2}}} \geq (mf_0)^{{s \choose 2}/{r \choose 2}},$$
where the second to last inequality follows from $(1+b)(1-b/4) \geq 1+b/2$ for any $b \in [0,1]$.
\end{proof}

The previous theorem easily implies a general lower bound for the largest value of $\gc{S}$.

\begin{theorem}\label{completelowerbound}
If $m \geq m_0$, there is some $S \subseteq R$ of size $s$ with $\gc{S} \geq (mf_0)^{{s \choose 2}/{r \choose 2}}$.
\end{theorem}
\begin{proof}
By the previous theorem, either such an $S$ exists or $\prod_{S \subseteq R} \gc{S} \geq (mf_0)^{r-2 \choose s-2}$. In this latter case, since this is the product of ${r \choose s}$ numbers, there must be some $S$ with $\gc{S} \geq (mf_0)^{{r-2 \choose s-2}/{r \choose s}} = (mf_0)^{{s \choose 2}/{r \choose 2}}$.
\end{proof}

Before we proceed, note that:
$$d = {r-2 \choose s-1}/{r-2 \choose s-2} = \frac{r-s}{s-1}.$$

We now simply rewrite the statement of the previous theorem in more familiar notation.

\begin{theorem}
Every Gallai coloring of a complete graph on $m$ vertices has a vertex subset using at most $s$ colors of order $\Omega\left(m^{{s \choose 2}/{r \choose 2}}\log^{c_{r,s}}m\right)$, where

$$c_{r,s}=
\left\{
\begin{array}{cl}
s(r-s) & \mbox{if } s < r-1,\\
1 & \mbox{if } s=r-1 \mbox{ and r is even},\\
(r+3)/r & \mbox{if } s=r-1 \mbox{ and r is odd}.
\end{array}\right.
$$
\end{theorem}
\begin{proof}
If $s < r-1$ and $m \geq m_0$, by Theorem \ref{completelowerbound} the coloring has a subchromatic set of order at least $m^{{s \choose 2}/{r \choose 2}}\left(c\log^2m\right)^{{s \choose 2}d}$. As $2{s \choose 2} d = s(s-1) \frac{r-s}{s-1} = s(r-s)$, this gives the desired bound in this case.

If $s = r-1$, $r$ is even, and $m \geq m_0$, by Theorem \ref{completelowerbound}, the coloring has a subchromatic set of order at least $m^{{s \choose 2}/{r \choose 2}}\left(c\log^2m\right)^{(r/2)d}$. As $$2(r/2) {s \choose 2}{r \choose 2}^{-1} d = r \frac{s(s-1)}{r(r-1)} \frac{r-s}{s-1} = r \frac{(r-1)(r-2)}{r(r-1)}\frac{1}{r-2} = 1,$$ this gives the desired bound in this case.

If $s = r-1$, $r$ is odd, and $m \geq m_0$, by Theorem \ref{completelowerbound}, the coloring has a subchromatic set of order at least $m^{{s \choose 2}/{r \choose 2}}\left(c\log^2m\right)^{((r+3)/2)d}$. As $$2((r+3)/2) {s \choose 2}{r \choose 2}^{-1} d = (r+3) \frac{s(s-1)}{r(r-1)} \frac{r-s}{s-1} = (r+3) \frac{(r-1)(r-2)}{r(r-1)}\frac{1}{r-2} = (r+3)/r,$$ this gives the desired bound in this case.
\end{proof}

\appendix
\section{Proof of Lemma~\ref{factsaboutf}}
\label{appendixA}
For convenience, we restate both the definition of $f$ and the statement of the lemma here:

$$f(n):=
\left\{
\begin{array}{lll}
c\log^2(Cn) & \mbox{if } 0 < n \leq m^{4/9}\\
c^2\log^2(m^{4/9})\log^2(Cnm^{-4/9}) & \mbox{if } m^{4/9} < n \leq m^{8/9}\\
c^3\log^4(m^{4/9})\log^2(Cnm^{-8/9}) & \mbox{if } m^{8/9} < n \leq m,
\end{array} \right., $$
where $D=2^{2048}$, $C=2^{D^8},$ and $c=\log^{-2}(C^2) = D^{-16}/4$.

\begin{lemma}
If $m \geq C$, then the following statements hold about $f$ for any integer $n$ with $1 < n \leq m$.
\begin{enumerate}
\item For any $\alpha \in \left[\frac{1}{n},1\right], \, f(\alpha n) \geq \alpha f(n)$.
\item For any $\alpha_1,\alpha_2,\alpha_3 \in \left[ \frac{1}{n},1 \right]$ such that $\sum_i \alpha_i = 1$ we have, taking $n_i = \alpha_i n$, $$nf(n) - \sum_i n_i f(n_i) \leq \frac{8}{\log C}nf(n).$$
\item For $i \geq 0$ and $m^{7/18} \geq 2^j \geq 1$ we have $f(2^i)\log^2(D2^j) \geq 512 f(2^{i+\frac{8}{7}j})$.
\item For $1 \leq \tau \leq n \leq D^3 \tau$, we have $f(\tau) \geq f(n)/2$.
\item For any $\alpha \in \left[\frac{1}{n},\frac{1}{32}\right], \, f(\alpha n) \geq 16 \alpha f(n)$.
\end{enumerate}
\end{lemma}
\begin{proof}
Observe that $f(n)$ has two points of discontinuity: $p_0 = m^{4/9}$ and $p_1 = m^{8/9}$. Recall that the three intervals of $f$ are $(0,p_0],(p_0,p_1],(p_1,m]$; name these $I_0,I_1,I_2$, respectively.

If $t$ is either $p_0$ or $p_1$, then we have $f_+(t):=\lim_{n \rightarrow t^+}f(n) \leq \lim_{n \rightarrow t^-}f(n)=:f_-(t)$.

Observe further that if $n$ is in some interval $I$ of $f$ then for any $t \in I$ we have $f(t) = \gamma \log^2(\delta t)$ for constants $\gamma,\delta$ with $\delta t \geq C$.

\noindent {\bf Proof of Fact 1:}
We first argue that it is sufficient to show Fact $1$ in the case that both $n$ and $\alpha n$ are in the same interval of $f$. Intuitively, the points of discontinuity only help us. If $n$ is in $I_1$, $\alpha n$ is in $I_0$, and we have shown Fact 1 holds when $n$ and $\alpha n$ are in the same interval, then 
$$f(\alpha n) \geq \frac{\alpha n}{p_0} f_+(p_0) \geq \frac{\alpha n}{p_0} f_-(p_0) \geq \frac{\alpha n}{p_0}\frac{p_0}{n}f(n) = \alpha f(n).$$
The case where $n$ is in $I_2$ and $\alpha n$ is in $I_1$ and the case where $n$ is in $I_2$ and $\alpha n$ is in $I_0$ hold by essentially the same argument.

We next show Fact $1$ in the case that both $n$ and $\alpha n$ are in the same interval $I$ of $f$. We have, choosing $\gamma$ and $\delta$ to be such that $f(t)=\gamma \log^2 (\delta t)$ on $I$, that
$$f(\alpha n) = \gamma \log^2(\alpha \delta n),$$
$$\alpha f(n) = \gamma \alpha \log^2(\delta n) = \gamma (\sqrt{\alpha}\log(\delta n))^2.$$
Thus, it is sufficient to show that $\log(\alpha \delta n) - \sqrt{\alpha}\log(\delta n) \geq 0$. Note equality holds if $\alpha = 1$. We consider the first derivative with respect to $\alpha$; we will show that it is negative for $\alpha \geq \frac{4}{\ln^2(\delta n)}$. The first derivative is:

$$\frac{1}{\alpha \ln(2)} - \frac{1}{2\sqrt{\alpha}}\log(\delta n) = \frac{2 - \sqrt{\alpha}\ln(\delta n)}{\alpha \ln(4)}.$$

Note the above is negative if $2 - \sqrt{\alpha}\ln(\delta n) \leq 0$, which is equivalent to $\alpha \geq \frac{4}{\ln^2(\delta n)}$.

Therefore, for $\alpha \in \left[\frac{4}{\ln^2(\delta n)},1\right]$, assuming $\alpha n \in I$, we have $f(\alpha n) \geq \alpha f(n)$. If $\alpha < \frac{4}{\ln^2(\delta n)}$ with $\alpha n \in I$ then, 
$$\alpha f(n) < \frac{4}{\ln^2(\delta n)} \gamma \log^2(\delta n) = \frac{4}{\ln^2(2)}  \gamma \leq \log^2C \gamma \leq \gamma \log^2(\delta \alpha n) = f(\alpha n),$$
where the first inequality follows by the assumed upper bound on $\alpha$ and the last one by $\alpha n$ in $I$ (and so $\delta \alpha n \geq C$).

\noindent {\bf Proof of Fact 2:}
Let $\gamma, \delta$ be such that for $t$ in the interval $I_j$ containing $n$ we have $f(t) = \gamma \log^2(\delta t)$. We define a new function $f_2$ whose domain is $[n \log^{-1}C,n]$. For any $t$ in the domain of $f_2$ that is in $I_j$, we define $f_2(t) = f(t)$, and for any $t$ in the domain of $f_2$ that is not in $I_j$, we define $f_2(t) = \gamma \log^2C$. If there is some point $t$ in $[n\log^{-1}C ,n]$ that is not in $I_j$, then $t$ must be in $I_{j-1}$, as $t \geq n \log^{-1}C \geq p_{j-1} \log^{-1}C > p_{j-2}$. Then note we have chosen, for $t$ not in $I_j$, $f_2(t)=f_+(p_{j-1})$. Therefore, $f_2$ is continuous. Also, $tf_2(t)$ is convex (this is easy to see by looking at the first derivative). The main idea behind the proof will be to replace $f$ by $f_2$ and then apply convexity to get the bounds.

We claim $f(t) \geq f_2(t)$ for all $t$ in the domain of $f_2$. If $t$ is in $I_j$, then $f_2(t)=f(t)$. Otherwise, $t$ is in $I_{j-1}$. For any $t \in [n \log^{-1}C,n]$, note that $\delta t \geq \delta n \log^{-1}C \geq C \log^{-1}C$. Therefore,
$$f(t) = \frac{\gamma}{c \log^2(m^{4/9})} \log^2(\delta t m^{4/9}) \geq \frac{\gamma}{c \log^2(m^{4/9})} \log^2\left(m^{4/9}C\log^{-1}C\right) \geq \frac{\gamma}{c} \geq \gamma \log^2C = f_2(t),$$
where the first equality follows by $t \in I_{j-1}$ and the first inequality by $\delta t \geq \log^{-1}C$.

Take $S = \{ i : \alpha_i \geq \log^{-1}C\}$. For $i \in S$ we have $\alpha_i n$ is in the domain of $f_2$. Take $\kappa$ such that $\sum_{i \in S} \alpha_i = \kappa$. Since $\sum_i \alpha_i = 1$, $\kappa = 1-\sum_{i \not \in S} \alpha_i \geq 1-3 \log^{-1}C$. Hence,

$$\sum_i n_if(n_i) \geq \sum_{i \in S} n_if(n_i) \geq \sum_{i \in S} n_if_2(n_i) \geq \sum_{i \in S}\frac{\kappa}{\size{S}}nf_2\left(\frac{\kappa}{\size{S}}n\right) =$$
$$\kappa n f_2\left(\frac{\kappa}{\size{S}}n\right) \geq \kappa n f_2\left(\frac{1-3\log^{-1}C}{3}n\right) \geq \kappa n f_2(n/4) \geq \kappa \gamma n \log^2(\delta n/4),$$
where the third inequality follows Jensen's inequality applied to the convex function $tf_2(t)$ and the fourth inequality holds since $f_2$ is an increasing function.

This gives

$$\kappa n f(n) - \sum_{i \in S} n_if(n_i) \leq \kappa \gamma n \log^2(\delta n) - \kappa \gamma n \log^2(\delta n/4) = \kappa \gamma n (\log^2(\delta n) - \log^2(\delta n/4)).$$

We now consider
$$\log^2(\delta n) - \log^2(\delta n/4) = (\log(\delta n) + \log(\delta n/4))(\log(\delta n)-\log(\delta n/4)) \leq 2 \log (\delta n) \log 4 = 4 \log(\delta n).$$

Noting that $\log(\delta n) \geq \log C$, we get 
$$\kappa \gamma n (4 \log(\delta n)) \leq \frac{4}{\log C} \gamma n\log^2(\delta n) = \frac{4}{\log C} \gamma n f(n).$$

Thus,

$$nf(n) - \sum_i n_i f(n_i) \leq (1 - \kappa)nf(n) + \kappa nf(n) - \sum_{i \in S} n_if(n_i) \leq$$ 
$$\frac{1}{2 \log C}nf(n) + \frac{4}{\log C}nf(n) \leq \frac{8}{\log C}nf(n).$$

\noindent {\bf Proof of Fact 3:}
Take $\gamma,\delta$ such that for $t$ in the interval of $f$ containing $2^i$ we have $f(t) = \gamma \log^2(\delta t)$. Take $j' = \frac{8}{7}j$. If $2^i$ and $2^{i + j'}$ are in the same intervals of $f$, then we get
$$f(2^i)\log^2(D2^j) = \gamma \log^2(\delta 2^i)\log^2(D2^j) = \gamma \log^2(2^{i + \log\delta})\log^2(2^{j+\log D}) =$$ 
$$\gamma \log^2(2^{(j + \log D)(i+\log\delta)}) \geq 512 \gamma \log^2(\delta 2^{i + j'}) = 512f(2^{i+\frac{8}{7}j}),$$
where the last inequality follows from $i + \log\delta \geq \log C \geq \log D$ and $j + \log D \geq \log D$. Therefore, $$(j+\log D)(i + \log\delta) \geq \frac{\log D}{2}(j+i+\log\delta) \geq \frac{\log D}{4}(2j + i + \log\delta) \geq 512(j' + i + \log\delta).$$

If $2^i$ and $2^{i + j'}$ are in different intervals of $f$, then they are in adjacent intervals since $2^{j'} \leq m^{4/9}$. Therefore,
$$f(2^{i+\frac{8}{7}j}) = f(2^{i + j'}) = c\gamma \log^2(m^{4/9}) \log^2(\delta m^{-4/9} 2^{i + j'}) \leq c\gamma \log^2(\delta 2^i)\log^2(2^{j'}) \leq$$ 
$$2c \left(\gamma \log^2(\delta 2^i)\right) \log^2(2^j) \leq \frac{1}{512}f(2^i)\log^2(2^j) \leq \frac{1}{512}f(2^i)\log^2(D2^j),$$
where the first inequality follows by the fact that if $a_0 \geq a_1 \geq b_1 \geq b_0 \geq 2$ and if $a_0b_0=a_1b_1$ then $(\log a_0)(\log b_0) \leq (\log a_1)(\log b_1)$. To see this last fact about logarithms, one may take the logarithm of both sides and apply the concavity of the logarithm function.

\noindent {\bf Proof of Fact 4:}
Choose $\gamma,\delta$ such that for $t$ in the same interval $I_j$ as $\tau$ we have $f(t) = \gamma \log^2(\delta t)$. If $n$ and $\tau$ are in different intervals of $f$, then $n$ must be in $I_{j+1}$ as $n/\tau \leq D^3 < m^{4/9}$. Furthermore, for $n$ and $\tau$ to be in different intervals, we must have $D^3 \delta \tau \geq Cm^{4/9}$ and so $\delta \tau \geq m^{4/9}$. This gives:
$$f(\tau) = \gamma \log^2(\delta \tau) \geq \gamma \log^2(m^{4/9}) \geq c \log^2(m^{4/9}) \gamma \log^2(CD^3) \geq $$
$$c \log^2(m^{4/9}) \gamma \log^2(\delta m^{-4/9} n) = f(n),$$
where the second inequality follows by $CD^3 \leq C^2$ and $c = \log^{-2}(C^2)$.

Otherwise, $\tau$ and $n$ are in the same interval. Then 
$$f(n) = \gamma \log^2(\delta n) \leq \gamma \log^2(\delta D^3 \tau) \leq \gamma \log^2 ((\delta \tau)^{4/3}) \leq 2 \gamma \log^2(\delta \tau) = 2f(\tau),$$
where the second inequality follows by $\delta \tau \geq C$ and $C^{1/3} \geq D$.

\noindent {\bf Proof of Fact 5:}
Let $\alpha \leq \frac{1}{32}$ be given. Consider:
$$16 \alpha f(n) = \frac{1}{2}(32 \alpha f(n)) \leq \frac{1}{2} f(32\alpha n) \leq f(\alpha n),$$
where the first inequality is by the first fact about $f$ and the second inequality is by the fourth fact about $f$.

\end{proof}

\section{Proof of Lemma \ref{weightedgraphs} on discrepancy in weighted graphs}
\label{appendixB}

We argue that if a weighted graph on $r$ vertices deviates in structure from the complete graph with edges of equal weight and if $s < r-1$, then there is some set of vertices $S$ of size $s$ so that the sum of the weights of the edges contained in $S$ is substantially larger than average.

\begin{lemma}
Given weights $w_P$ for $P \in {R \choose 2}$ with $w_P \geq 0$, take $w = \sum_P w_P$. Then if $s < r-1$, if some $w_P$ differs from ${r \choose 2}^{-1}w$ by at least $F$, then there is some $S \subseteq R$ of size $s$ satisfying $\sum_{P \subseteq S} w_P \geq \frac{{s \choose 2}}{{r \choose 2}}w + \frac{{s \choose 2}F}{r {r \choose 2}^2}$.
\end{lemma}
\begin{proof}
We will directly handle the case $s=r-2$, from which the other cases will follow. We are interested in finding an $S \subseteq R$ of size $r-2$ with a large value for the total weight of edges in $S$. For each $S$ we give this value a name: $Z_S = \sum_{P \subseteq S} w_P$. Note that $Z_S$ is closely related the following: for $Q \in {R \choose 2}$, define $Y_Q$ to be the weight of edges incident to at least one vertex of $Q$: $Y_Q = \sum_{P \in {R \choose 2}:P \cap Q \neq \emptyset}w_P$. Then, if we take $S$ to be $R \setminus Q$, we have $Y_Q + Z_S$ is the total weight of all the edges. Thus, to show that there is a large $Z_S$, it is sufficient to show that there is a small $Y_Q$. Towards this end, choose $Q \in {R \choose 2}$ uniformly at random; we will now compute the variance of $Y_Q$.

Take, for $P \in {R \choose 2}$, $X_P$ to be $w_P$ if $P \cap Q \neq \emptyset$ and $0$ otherwise. Then, taking $w = \sum_P w_P$, we get
$$\E{X_P} = \Pr[P \cap Q \neq \emptyset]w_P = \frac{2r-3}{{r \choose 2}}w_P.$$
By linearity of expectation, we have
$$\E{Y_Q} = \sum_{P \in {R \choose 2}} \E{X_P} = \frac{2r-3}{{r \choose 2}} \sum_{P \in {R \choose 2}} w_P = \frac{2r-3}{{r \choose 2}}w,$$
and
$$\E{Y_Q^2} = \sum_{P \in {R \choose 2}} \sum_{P' \in {R \choose 2}} \E{X_PX_{P'}} =$$
$$\sum_{P \in {R \choose 2}}\E{X_P^2} + \sum_{v \in R} \sum_{(P,P') \in {R \choose 2}^2 : v \in P,v \in P', P \neq P'} \E{X_PX_{P'}} + \sum_{(P,P') \in {R \choose 2}^2 : P \cap P' = \emptyset}\E{X_PX_{P'}},$$
where the last equality follows by partitioning the pairs $P,P'$ into those which are equal, those which are distinct but intersect in some vertex $v$, and those which are disjoint.

We now look at these terms individually.
$$\E{X_P^2} = \Pr[P \cap Q \neq \emptyset]w_P^2 = \frac{2r-3}{{r \choose 2}} w_P^2.$$
For $P=\{v,u\},P'=\{v,u'\}$ distinct and intersecting, the event $P\cap Q \neq \emptyset$ and $P' \cap Q \neq \emptyset$ can occur if either $v \in Q$ or $Q=\{u,u'\}$; the first of these has probability $\frac{r-1}{{r \choose 2}}$ and the second has probability $\frac{1}{{r \choose 2}}$, and they are disjoint events. So, if $P$ and $P'$ intersect in a vertex we get:
$$\E{X_PX_{P'}}=\frac{r}{{r \choose 2}}w_Pw_{P'}.$$
If $P,P'$ are disjoint, then for $X_PX_{P'}$ to be non-zero we must have that $Q$ has an element from $P$ and an element from $P'$, which occurs with probability $\frac{4}{{r \choose 2}}$, so in this case:
$$\E{X_PX_{P'}}=\frac{4}{{r \choose 2}}w_Pw_{P'}.$$

Therefore, taking for $v \in R$ the (weighted) degree $d(v)$ to be $\sum_{P \in {R \choose 2}: v \in P} w_P$, $\E{Y_Q^2}$ is equal:

\begin{eqnarray*}
 & & \sum_{P \in {R \choose 2}} \frac{2r-3}{{r \choose 2}}w_P^2 + \sum_{v \in R} \sum_{(P,P') \in {R \choose 2}^2 : v \in P, v\in P', P \neq P'}\frac{r}{{r \choose 2}}w_Pw_{P'} + \sum_{(P,P') \in {R \choose 2} : P \cap P' = \emptyset}\frac{4}{{r \choose 2}}w_Pw_{P'} \\
& = & \frac{2r-3}{{r \choose 2}}\sum_{P \in {R \choose 2}} w_P^2 + \frac{r}{{r \choose 2}}\sum_{v \in R} \left( d(v)^2 - \sum_{P \in {R \choose 2} : v \in P} w_P^2 \right) + \frac{4}{{r \choose 2}}\left(\left(\sum_{P \in {R \choose 2}}w_P\right)^2 - \sum_{v \in R} d(v)^2 + \sum_{P \in {R \choose 2}}w_P^2\right)\\
& = & \frac{2r-3}{{r \choose 2}}\sum_{P \in {R \choose 2}}w_P^2 + \frac{r}{{r \choose 2}}\sum_{v \in R} d(v)^2 - 2\frac{r}{{r \choose 2}}\sum_{P \in {R \choose 2}}w_P^2 + \frac{4}{{r \choose 2}}w^2 - \frac{4}{{r \choose 2}}\sum_{v \in R}d(v)^2 + \frac{4}{{r \choose 2}}\sum_{P \in {R \choose 2}}w_P^2 \\
& = & \frac{4}{{r \choose 2}}w^2 + \frac{1}{{r \choose 2}}\sum_{P \in {R \choose 2}}w_P^2 + \frac{r-4}{{r \choose 2}} \sum_{v \in R}d(v)^2
\end{eqnarray*}

Note that $\sum_{v \in R} d(v)^2$ is minimized subject to the constraint $\sum_{v \in R}d(v) = 2w$ when the $d(v)$ are pairwise equal by the Cauchy-Schwarz inequality, so $\sum_{v \in R} d(v)^2 \geq \sum_{v \in R} \left(\frac{2w}{r}\right)^2$, so the above is at least

$$\frac{4}{{r \choose 2}}w^2 + \frac{1}{{r \choose 2}}\sum_{P \in {R \choose 2}}w_P^2 + \frac{r-4}{{r \choose 2}} \sum_{v \in R}\left(\frac{2w}{r}\right)^2 = $$
$$\frac{4}{{r \choose 2}}w^2 + \frac{1}{{r \choose 2}}\sum_{P \in {R \choose 2}}w_P^2 + r\frac{r-4}{{r \choose 2}} \left(\frac{2w}{r}\right)^2 = \frac{8r-16}{r{r \choose 2}}w^2 + \frac{1}{{r \choose 2}}\sum_{P \in {R \choose 2}}w_P^2.$$
The variance of $Y_Q$ satisfies:
$$\Var{Y_Q} = \E{Y_Q^2}-\E{Y_Q}^2 \geq \frac{8r-16}{r{r \choose 2}}w^2 + \frac{1}{{r \choose 2}}\sum_{P \in {R \choose 2}}w_P^2 - \left(\frac{2r-3}{{r \choose 2}}w\right)^2 = \frac{1}{{r \choose 2}}\sum_{P \in {R \choose 2}}w_P^2-\frac{1}{{r \choose 2}^2}w^2.$$

Note we may rewrite the variance as:
$$\Var{Y_Q} \geq \frac{1}{{r \choose 2}}\sum_{P \in {R \choose 2}}w_P^2-\frac{1}{{r \choose 2}^2}w^2 = \Var{w_Q}.$$

If some $w_P$ is far from $w/{r \choose 2} = \E{w_Q}$, then the variance will be large. Assume that for some $P'$ there is a non-zero real $F$ so that $w_{P'} = w/{r\choose 2} + F$. Note $\Var{Y_Q} \geq \Var{w_Q} = \E{(w_Q - w/{r \choose 2}}$ and that $(w_Q-w/{r \choose 2})^2$ is a non-negative random variable. If $Q = P'$ (which occurs with probability ${r \choose 2}^{-1}$) then this random variable has value $F^2$, so its expectation is at least ${r \choose 2}^{-1}F^2$. That is, the variance of $w_Q$ is at least ${r \choose 2}^{-1}F^2$.

Thus, there must be some $Q$ so that 
$$\size{Y_Q - \frac{2r-3}{{r \choose 2}}w} \geq {r \choose 2}^{-1/2}F \geq F/r.$$ 
If $Y_Q - \frac{(2r-3)w}{{r \choose 2}} \geq F/r$, since there are ${r \choose 2}$ different $Y_Q$ and the average is $\frac{(2r-3)w}{{r \choose 2}}$, there must be some $Q'$ so that $$Y_{Q'}-\frac{2r-3}{{r \choose 2}}w \leq \frac{-F}{r\left({r \choose 2}-1\right)} \leq -\left(r{r \choose 2}\right)^{-1}F.$$
The other case is that 
$$Y_Q - \frac{2r-3}{{r \choose 2}}w \leq -F/r \leq -\left(r{r \choose 2}\right)^{-1}F.$$ Therefore, there is some $Q$ with $Y_Q \leq \frac{(2r-3)w - F/r}{{r \choose 2}}$.

Define $S = R \setminus Q$. We get $Z_S + Y_Q = w$ so $Z_S = w-Y_Q$. By the above, there is some $S$ with 
$$Z_S \geq w - \frac{(2r-3)w - F/r}{{r \choose 2}} = \frac{{r-2 \choose 2}w+F/r}{{r \choose 2}}.$$

Taking $S$ as above, choosing a random $S' \in {S \choose s}$, we get that $\E{Z_{S'}} \geq \frac{{s \choose 2}}{{r-2 \choose 2}}\frac{{r-2 \choose 2}w+F/r}{{r \choose 2}}$. Therefore, there must be some $S' \in {S \choose s}$ with 
$$Z_{S'} \geq \frac{{s \choose 2}}{{r \choose 2}}w + \frac{{s \choose 2}F}{r{r \choose 2}{r - 2 \choose 2}} \geq \frac{{s \choose 2}}{{r \choose 2}}w + \frac{{s \choose 2}F}{r {r \choose 2}^2}.$$
\end{proof}

The case where $s < r-1$ in Lemma \ref{weightedgraphs} is an immediate corollary.

\begin{lemma}
Given weights $w_P$ for $P \in {R \choose 2}$ with $w_P \geq 0$, take $w = \sum_P w_P$. If $s < r-1$, then either there are at least ${r \choose 2}$ pairs $P$ (i.e. all of them) with $w_P > 0$ or there is some $S \subseteq R$ of size $s$ satisfying $\sum_{P \subseteq S} w_P \geq \left(1 + \left(4r {r \choose 2}^2\right)^{-1}\right)\frac{{s \choose 2}}{{r \choose 2}}w$.
\end{lemma}
\begin{proof}
Assume there is some $P'$ with $w_{P'} = 0$. We may apply the previous lemma with $F=w/{r \choose 2}$, since $w_{P'}$ differs from $w/{r \choose 2}$ by $F$. This gives that there is some set $S \subseteq R$ of size $s$ satisfying:

$$\sum_{P \subseteq S} w_P \geq \left(\frac{{s \choose 2}}{{r \choose 2}} + \frac{{s \choose 2}}{r {r \choose 2}^3}\right)w \geq \left(1 + \left(4r {r \choose 2}^2\right)^{-1}\right)\frac{{s \choose 2}}{{r \choose 2}}w.$$
\end{proof}

The following lemma states that if in a weighted graph there is a vertex whose degree deviates from the average, then there is a set $S \subseteq R$ of size $r-1$ so that the sum of the weights of the edges contained in $S$ is substantially larger than average.

\begin{lemma} \label{degrees}
Given weights $w_P$ for $P \in {R \choose 2}$ with $w_P \geq 0$, take $w = \sum_P w_P$. For $v \in R$, define $d(v) := \sum_{P : v \in P} w_P$. If there is some $v \in R$ for which $d(v)$ differs from $2w/r$ by at least $F$, then there is some $S \subseteq R$ of size $r-1$ with $\sum_{P \subseteq S} w_P \geq \frac{{r-1 \choose 2}}{{r \choose 2}}w + F/r.$
\end{lemma}
\begin{proof}
Choose a vertex $v$ for which $\size{d(v) - 2w/r} \geq F$. If $d(v) \leq 2w/r - F$, then we may take $S = V \setminus \{v\}$. This gives:
$$\sum_{P \subseteq S} w_P = w - d(v) \geq w-(2w/r - F) = \frac{{r-1 \choose 2}}{{r \choose 2}}w + F \geq \frac{{r-1 \choose 2}}{{r \choose 2}}w + F/r.$$
Otherwise, we have $d(v) \geq 2w/r + F$. Since $\sum_u d(u) = 2w$,
$$\sum_{u \neq v} d(u) \leq 2w - (2w/r+F) = ((r-1)/r)2w-F.$$
Since the average is $2w/r$, there is some $u$ with 
$$d(u) \leq \left(\frac{r-1}{r}2w-F\right)/(r-1) = 2w/r-F/(r-1) \leq 2w/r-F/r.$$ We may take $S = V \setminus \{u\}$. We get:
$$\sum_{P \subseteq S} w_P = w - d(u) \geq w-(2w/r - F/r) = \frac{{r-1 \choose 2}}{{r \choose 2}}w + F/r.$$
\end{proof}

A strengthening of the case $s=r-1$ and $r$ is even in Lemma \ref{weightedgraphs} is a corollary.

\begin{lemma}
Given weights $w_P$ for $P \in {R \choose 2}$ with $w_P \geq 0$, take $w = \sum_P w_P$. Either there are at least $r/2$ pairs $P$ for which $w_P \geq w/r^2$ or there is some $S \subseteq R$ of size $r-1$ with $\sum_{P \subseteq S} w_P \geq \left(1 + \left(4r {r \choose 2}^2\right)^{-1}\right)\frac{{r-1 \choose 2}}{{r \choose 2}}w.$
\end{lemma}
\begin{proof}
If there are fewer than $r/2$ pairs $P$ for which $w_P \geq w/r^2$, then we must have that there is some vertex $v$ not adjacent to any such pair. For this $v$, $d(v) \leq (r-1)w/r^2 \leq w/r$. The previous lemma gives that there is some set $S$ of size $r-1$ with:
$$\sum_{P \subseteq S} w_P \geq \left(\frac{{r-1 \choose 2}}{{r \choose 2}} + \frac{1}{r^2}\right)w \geq \left(1 + \left(4r {r \choose 2}^2\right)^{-1}\right)\frac{{r-1 \choose 2}}{{r \choose 2}}w.$$
\end{proof}

Finally, we prove a strengthening of the case $s=r-1$ and $r$ is odd in Lemma \ref{weightedgraphs}:

\begin{lemma}
Given weights $w_P$ for $P \in {R \choose 2}$ with $w_P \geq 0$, take $w = \sum_P w_P$. If $r$ is odd either there are at least $(r+3)/2$ pairs $P$ for which $w_P > 0$ or there is some $S \subseteq R$ of size $r-1$ with $\sum_{P \subseteq S} w_P \geq \left(1 + \left(4r {r \choose 2}^2\right)^{-1}\right)\frac{{r-1 \choose 2}}{{r \choose 2}}w.$
\end{lemma}
\begin{proof}
Assume there is no $S \subseteq R$ of size $r-1$ with $\sum_{P \subseteq S} w_P \geq \left(1 + \left(4r {r \choose 2}^2\right)^{-1}\right)\frac{{r-1 \choose 2}}{{r \choose 2}}w.$ In this case there is no $S \subseteq R$ of size $r-1$ with $\sum_{P \subseteq S} w_P \geq \frac{{r-1 \choose 2}}{{r \choose 2}}w + w/(4r^3)$, as this latter term is larger than $\left(1 + \left(4r {r \choose 2}^2\right)^{-1}\right)\frac{{r-1 \choose 2}}{{r \choose 2}}w$.

We define an unweighted graph $G=(V,E)$ by taking $V=R$ and a possible edge $e \in {R \choose 2}$ is in $E$ if and only if $w_e \geq  w/(4 r^3)$. By the previous lemma, $G$ has at least $r/2$ edges $e$ satisfying $w_e \geq w/r^2$, and, indeed, the proof of the previous lemma shows that every vertex must have degree at least $1$. Since $r$ is odd, $G$ must have at least $(r+1)/2$ edges $e$ with $w_e \geq w/r^2$, and so it must have some vertex $v$ incident to two such edges. 

Fix two neighbors $v_1,v_2$ of $v$ so that $w_{\{v,v_1\}}$ and $w_{\{v,v_2\}}$ both have weight at least $w/r^2$. We claim that both $v_1$ and $v_2$ have degree at least $2$ in $G$. Assume at least one of them, without loss of generality $v_1$, has degree one. We must have $d(v) \leq 2w/r + w/(2r^2)$, for otherwise we have a contradiction by Lemma \ref{degrees}. However, this gives that, since $w_{\{v,v_2\}} \geq w/r^2$, we must have $w_{\{v,v_1\}} \leq d(v) - w_{\{v,v_2\}} \leq 2w/r - w/(2r^2)$. Then all other $P$ incident to $v_1$ have weight at most $w/(4r^3)$, so 
$$w_{\{v,v_1\}} \leq w_{\{v,v_1\}} + (r-2) w/(4r^3) \leq 2w/r-w/(2r^2) + rw/(4r^3) = 2w/r-w/(4r^2).$$
Then by Lemma \ref{degrees} we have reached a contradiction.

Therefore, we must have that there are at least $3$ vertices of degree $2$ in $G$ and that every vertex has degree at least $1$. Then the sum of the degrees is at least $6 + (r-3) = r+3$ and so the number of edges of $G$ must be at least $(r+3)/2$, as desired.
\end{proof}

\section{Proof of Lemma~\ref{factsaboutfepsilon}}
\label{appendixC}
For convenience, we restate both the lemma and definition of $f_\epsilon$ here:

$f_\epsilon(\ell) = (\log m)^{C\left(\frac{\log(\alpha m^\epsilon)}{\log m}\right)}$, where $\alpha = \ell/n$. Recall also $m_0$ from Equation \ref{definitionm}.

\begin{lemma}
The following statements hold about $f_\epsilon$ for every choice of $\epsilon \geq 0$, $n > 1$, and $m \geq m_0$.
\begin{enumerate}
\item For any $\alpha\in [\frac{1}{n},1]$, 
$$f_\epsilon(\alpha n) \geq \alpha^{1/\left(2 {r-2 \choose s-2}\right)} f_\epsilon(n).$$ 
In particular, $f_\epsilon(\alpha n) \geq \alpha f_\epsilon(n)$.
\item For any $\alpha_1,\alpha_2,\alpha_3 \in [\frac{1}{n},1]$ with $\alpha_1+\alpha_2+\alpha_3=1$, taking $n_i = \alpha_i n$,
$$nf_\epsilon(n) \leq \sum_i n_if_\epsilon(n_i) + 3(\log^{-3/4}m)nf_\epsilon(n).$$
\item For $i \geq 0$ and $m^{\delta} \geq 2^j \geq 1$ we have $f_\epsilon(2^i)\log^{2/{r-2 \choose s-2}}((\log^{1/4}m)2^j) \geq 256{r \choose 2} f_\epsilon(2^{i+2j})$.
\item For any $\alpha \geq \log^{-1}m$, $f_\epsilon(\alpha n) \geq f_\epsilon(n)/2$.
\end{enumerate}
\end{lemma}
\begin{proof}

\noindent {\bf Proof of Fact 1:}
Note 
$$f_\epsilon(\alpha n) = (\log m)^{C\left(\epsilon + \frac{\log \alpha}{\log m}\right)} = (\log m)^{C\frac{\log \alpha}{\log m}}(\log m)^{C\epsilon} = \alpha^{C \frac{\log\log m}{\log m}}f_\epsilon(n).$$
Since $0 < \alpha \leq 1$, it is sufficient to show that $C (\log\log m)/\log m \leq \left(2{r-2 \choose s-2}\right)^{-1}$. This holds because $C$, $\log\log m$ and $2{r-2 \choose s-2}$ are at most $\log^{1/3}m$, since $m \geq m_0$.

\noindent {\bf Proof of Fact 2:}
Note that the function $\ell f_\epsilon(\ell)$ is a convex function of $\ell$. Indeed, 
$$f_\epsilon(\ell) = (\log m)^{C\epsilon}(\log m)^{(\log \ell)/\log m} = (\log m)^{C \epsilon}\ell^{(\log\log m)/\log m}.$$
That is, $f_\epsilon(\ell)$ is a polynomial of degree greater than $0$ in $\ell$, so $\ell f_\epsilon(\ell)$ is a polynomial of degree greater than $1$ in $\ell$ and so is convex.

Take $S = \{ i : \alpha_i \geq \log^{-3/4} m\}$. Take $\kappa$ such that $\sum_{i \in S} \alpha_i = \kappa$. Note $\kappa \geq 1-2\log^{-3/4} m$.

$$\sum_i n_if_\epsilon(n_i) \geq \sum_{i \in S} n_if_\epsilon(n_i) \geq \sum_{i \in S}\frac{\kappa}{\size{S}}nf_\epsilon\left(\frac{\kappa}{\size{S}}n\right) =$$
$$\kappa n f_\epsilon \left(\frac{\kappa}{\size{S}}n\right) \geq \kappa n f_\epsilon\left(\frac{1-2\log^{-1}m}{3}n\right) \geq \kappa n f_\epsilon(n/4),$$
where the second inequality follows by Jensen's inequality applied to the convex function $\ell f_\epsilon(\ell)$.

This gives

$$\kappa n f_\epsilon(n) - \sum_{i \in S} n_if_\epsilon(n_i) \leq \kappa n f_\epsilon(n) - \kappa n f_\epsilon(n/4) = \kappa n (f_\epsilon(n) - f_\epsilon(n/4)).$$

We now consider
$$f_\epsilon(n)-f_\epsilon(n/4) = (\log m)^{C \epsilon} - (\log m)^{C \left(\epsilon + \frac{\log(1/4)}{\log m}\right)} = (\log m)^{C \epsilon}\left(1-(\log m)^{-2C/\log m}\right).$$

The second factor satisfies:
$$1-(\log m)^{-2C/\log m} = 1-2^{-2C(\log\log m)/\log m} \leq$$
$$1-(1-2C (\log\log m)/\log m) = 2 C (\log\log m)/\log m \leq \log^{-3/4} m,$$
where the first inequality follows by $2^x \geq 1+x$ for $x \leq 0$.

Thus,

$$nf_\epsilon(n) - \sum_i n_i f_\epsilon(n_i) \leq (1 - \kappa)nf_\epsilon(n) + \kappa nf_\epsilon(n) - \sum_{i \in S} n_if(n_i) \leq$$ 
$$2(\log^{-3/4}m)nf_\epsilon(n) + (\log^{-3/4}m)nf_\epsilon(n) \leq 3(\log^{-3/4}m)nf_\epsilon(n).$$

\noindent {\bf Proof of Fact 3:}
We prove a slightly stronger statement. Take $j' = j + \frac{1}{4}(\log\log m)$ so that $2^{j'} = (\log m)^{1/4}2^j$. We will show that 
$$f_\epsilon(2^i)\log^{2/{r-2 \choose s-2}}(2^{j'}) \geq f_\epsilon(2^{i + 2j'}).$$
This is indeed stronger than the original statement as $f_\epsilon(\ell)$ is an increasing function of $\ell$.

Consider

$$f_\epsilon(2^i) = (\log m)^{C\left(\epsilon + \frac{\log(2^i/n)}{\log m}\right)} = (\log m)^{C \epsilon}(\log m)^{C\frac{i}{\log m}}(\log m)^{-C\frac{\log n}{\log m}}.$$
Similarly,
$$f_\epsilon(2^{i+2j'}) = (\log m)^{C \epsilon}(\log m)^{C\frac{i + 2j'}{\log m}}(\log m)^{-C\frac{\log n}{\log m}}.$$
Therefore, it is sufficient to show that
$$(\log m)^{C\frac{i}{\log m}} \log^{2/{r-2 \choose s-2}}(2^{j'}) \geq 256{r \choose 2}(\log m)^{C\frac{i + 2j'}{\log m}},$$
or equivalently that 
$$\log^{2/{r-2 \choose s-2}}(2^{j'}) \geq 256{r \choose 2}(\log m)^{C\frac{2j'}{\log m}}.$$

Taking logarithms of both sides, we see that it is sufficient to have
$$2 (\log j')/{r-2 \choose s-2} \geq 2C j' (\log\log m)/(\log m) + \log\left(256{r \choose 2}\right),$$
or equivalently
$$2(\log j')/{r-2 \choose s-2} - 2C j'(\log\log m)/(\log m) - \log\left(256{r \choose 2}\right) \geq 0.$$

We consider the first derivative of this with respect to $j'$: it is $2(\ln(2)j')^{-1}/{r-2 \choose s-2} - 2C (\log\log m)/(\log m)$. Note this derivative is monotone decreasing for $j' \in [1,\infty)$, so the minimum of $2(\log j')/{r-2 \choose s-2} - 2Cj'(\log\log m)/(\log m) - \log\left(256{r \choose 2}\right)$ must be achieved at either the largest or smallest possible value of $j'$. We have assumed $m^{\delta}\log^{1/4}m \geq 2^{j'} \geq \log^{1/4}m$, so $2\delta\log m \geq \delta(\log m) + (\log\log m)/4 \geq j' \geq (\log\log m)/4$. We consider the two extrema.

If $j' = (\log\log m)/4$:
$$2\log((\log\log m)/4)/{r-2 \choose s-2} - \frac{C}{2}\log^2(\log m)/(\log m) - \log\left(256{r \choose 2}\right) \geq$$
$$2\log((\log\log m)/4)/{r-2 \choose s-2} - 1 - \log\left(256{r \choose 2}\right) \geq 0,$$
where the last inequality follows from $m \geq m_0$.

If $j' = 2\delta\log m$:
$$2\log(2\delta\log m)/{r-2 \choose s-2}-4\delta C (\log\log m) - \log\left(256{r \choose 2}\right) =$$
$$2\log(2\delta \log m)/{r-2 \choose s-2} - (\log\log m)/{r-2 \choose s-2} - \log\left(256{r \choose 2}\right) \geq 0,$$
where the last inequality follows from $m \geq m_0$.

\noindent {\bf Proof of Fact 4:}
Since $f_\epsilon$ is increasing, it is sufficient to show this for $\alpha = \log^{-1}m$. Then

$$f_\epsilon(\log^{-1} m) = (\log m)^{C\left( \epsilon - (\log\log m)/(\log m) \right)} =$$
$$(\log m)^{C \epsilon}2^{-(\log\log m)^2/\log m} \geq (\log m)^{C\epsilon}2^{-1} = f_\epsilon(n)/2.$$

\end{proof}

\end{document}